\documentclass[12pt]{article}
\usepackage{amsmath} 
\usepackage{amssymb} 
\usepackage{amsthm}
\usepackage{mathtools, booktabs}
\usepackage{url}
\usepackage{tikz-cd}
\usepackage{xparse}

\usepackage[pagewise]{lineno}
\theoremstyle{plain} 
\newtheorem{theorem}{Theorem}[section] 
\newtheorem{corollary}[theorem]{Corollary}
\newtheorem{proposition}[theorem]{Proposition}
\newtheorem{lemma}[theorem]{Lemma} 
\theoremstyle{definition} 
 \newtheorem{remark}[theorem]{Remark}

\newtheorem{nothing}[theorem]{}

\newcommand{\GsetA}{{}_G\operatorname{set}^A}
\newcommand{\Hom}{\operatorname{Hom}}

\newcommand{\kerlin}{\mathbb S_p}
\newcommand{\kerlint}{\mathbb S_2}
\newcommand{\ling}{\operatorname{Lin}_G}
\newcommand{\MonG}{\mathbb C B^{\mathbb C^\times}(G)}

\newcommand{\scp}{\mathbb S_p}

\newcommand{\ind}{\operatorname{Ind}}
\newcommand{\res}{\operatorname{Res}}
\newcommand{\defl}{\operatorname{Def}}
\newcommand{\infl}{\operatorname{Inf}}

\newcommand{\iso}{\operatorname{Iso}}
\newcommand{\out}{\operatorname{Out}}
\newcommand{\Aut}{\operatorname{Aut}}
\newcommand{\gl}{{\operatorname{GL}(2, p)}}

\usepackage{mathrsfs, enumitem} 

\title{The $p$-Biset Functor of Monomial Burnside Rings I: Composition Factors}

\author{İbrahim Kaan ASLAN$^1$ and Olcay COŞKUN$^2$}

\date{
$^1$ Indiana University, IN, USA \\ \texttt{ikaslan@iu.edu} \\ %
	$^2$ Center for Mathematical Research, ASOIU, Baku, Azerbaijan \\ 
   \texttt{olcay.coshkun@asoiu.edu.az}\\[2ex]%
}

\begin{document}
\maketitle

\begin{abstract}
We determine the complete list of composition factors, with multiplicities, of the monomial Burnside $p$-biset functor $\mathbb C B^{\mathbb{C}^\times}$ over the complex field $\mathbb K$ of characteristic 0. The method relies on a reduction to \emph{restriction kernels}, transforming the problem into the computation of certain $\mathbb C[\mathrm{Aut}(G)]$-modules attached to finite $p$-groups $G$. We describe these modules explicitly for all $G$ with non-trivial restriction kernel, and identify the corresponding simple functors. 
\end{abstract}

Keywords: monomial linearization; fibered biset functor; $p$-biset functor; composition factor.

\section{Introduction}\label{sec:intro}

The \emph{monomial Burnside ring} of a finite group \( G \), introduced by Dress \cite{Dress}, is the Grot\-hen\-dieck ring of the category of monomial \( G \)-sets. Similar to the role of permutation \( G \)-sets, monomial \( G \)-sets can be viewed as bases for \( \mathbb{C}G \)-modules by extending the \( G \)-action on a monomial \( G \)-set \( X \) to the vector space \( \mathbb{C}X \). The induced ring homomorphism 
\[
\ling \colon B^{\mathbb{C}^\times}(G) \to \mathcal{R}_{\mathbb{C}}(G),
\]
from the monomial Burnside ring \( B^{\mathbb{C}^\times}(G) \) to the representation ring \( \mathcal{R}_{\mathbb{C}}(G) \), is called the \emph{monomial linearization map}. This homomorphism is known to be surjective, as established by Brauer's Induction Theorem. Explicit sections were subsequently determined by Boltje \cite{Bo8}, Snaith \cite{Snaith}, and Symonds \cite{Sym}.

When the groups \( B^{\mathbb{C}^\times}(G) \) and \( \mathcal{R}_{\mathbb{C}}(G) \) are viewed as functors on suitable categories, the linearization map becomes a natural transformation between these functors. Using the Mackey functor framework, Boltje \cite{Bo8} determined a section for \( \ling \), referred to as \emph{canonical induction formula} for \( G \). Later, in \cite{BoCo}, Boltje and the second author extended this perspective by realizing \( \ling \) as a morphism between two \( \mathbb{C}^\times \)-fibered biset functors. They proved that \( \mathbb{C}\mathcal{R}_{\mathbb{C}} \) is a simple functor and that \( \mathbb{C}B^{\mathbb{C}^\times} \) serves as its projective cover. Additionally, Bouc \cite{Bouc0} demonstrated that \( \mathbb{C}\mathcal{R}_{\mathbb{C}} \) is a semisimple biset functor, while the fibered biset functor structure of \( \mathbb{C}B^{\mathbb{C}^\times} \) was determined in \cite{CoYi} for $p$-groups and in \cite{BoYi} generally.

For the rest, suppose $G$ is a $p$-group. The linearization map 
\[
\operatorname{Lin} \colon B(G) \to \mathcal{R}_{\mathbb{Q}}(G)
\]
from the Burnside ring \( B(G) \) to the rational representation ring \( \mathcal{R}_{\mathbb{Q}}(G) \) plays a crucial role in classifying endo-permutation modules. In \cite{BoTh}, Bouc and Th{\'e}venaz established the following short exact sequence of \( p \)-biset functors:
\[
\begin{tikzcd}
0 \arrow[r] & \mathbb{Q}D \arrow[r] & \mathbb{Q}B \arrow[r, "\operatorname{Lin}"] & \mathbb{Q}\mathcal{R}_{\mathbb{Q}} \arrow[r] & 0,
\end{tikzcd}
\]
where \( D(G) \) denotes the Dade group of a \( p \)-group \( G \). They further showed that both the image and kernel of \( \operatorname{Lin} \) are simple \( p \)-biset functors.

Similarly, in \cite{CoYi}, the second author and Yılmaz established a short exact sequence for the monomial linearization map. Specifically, they proved that the kernel of the monomial linearization map is the simple functor 
\[
\kerlin := S_{C_p, \{1\}, \{1\}, \mathbb{C}},
\]
yielding the exact sequence of fibered \( p \)-biset functors:
\[
\begin{tikzcd}
0 \arrow[r] & \kerlin \arrow[r] & \mathbb{C}B^{\mathbb{C}^\times} \arrow[r, "\operatorname{Lin}"] & \mathbb{C}\mathcal{R}_{\mathbb{C}} \arrow[r] & 0.
\end{tikzcd}
\]

Although \( \kerlin \) has not yet been explicitly identified with a known functor, its role as the kernel of the monomial linearization map makes it a significant object of study. In this paper, we analyze \( \kerlin \) as a \( p \)-biset functor and determine its composition factors explicitly. The general problem of determination of composition factors for biset functors is a challenging task. However, by applying \cite[Theorem 2.5]{BCK} (see Appendix D), we reduce this problem to determining the irreducible summands of the \emph{restriction kernels} of \( \kerlin \) at finite $p$-groups. Here, by a restriction kernel at a group $G$, we mean the submodule of all the elements in $\kerlin(G)$ which are annihilated by all morphisms mapping to a group of smaller order.

Our first main theorem identifies the \( p \)-groups for which the restriction kernel is nonzero:
\begin{theorem}\label{thm:main1}
    Let \( G \) be a finite \( p \)-group. The restriction kernel \( \mathcal{K}\kerlin(G) \) of \( \kerlin \) at \( G \) is nonzero if and only if \( G \) is either cyclic of order \( p^k \) (for some integer \( k \geq 1 \)) or is isomorphic to a direct product \( C_{p^k} \times C_p \) (for some integer \( k \geq 1 \)).
\end{theorem}

Using \cite[Theorem 2.5]{BCK}, there is a bijection between the set of composition factors of \( \kerlin \) as a \( p \)-biset functor and the set of irreducible constituents of the \( \mathbb{C}[\operatorname{Aut}(G)] \)-modules \( \mathcal{K}\kerlin(G) \), as \( G \) runs over the groups in the above theorem. The second main result of the paper determines these irreducible constituents, summarized as follows:
\begin{theorem}\label{thm:main}
The simple \( p \)-biset functor \( S_{G, V} \) is a composition factor of \( \kerlin \) if and only if one of the following statements hold:
\begin{enumerate}
    \item[\textnormal{(a)}] \( G \) has order \( p \), and \( V \) is the trivial \( \mathbb{C}[\operatorname{Aut}(G)] \)-module.
    \item[\textnormal{(b)}] \( G \) is cyclic of order \( p^k \) (for \( k \geq 2 \)), and \( V \) is any non-primitive irreducible \( \mathbb{C}[\operatorname{Aut}(G)] \)-module.
    \item[\textnormal{(c)}] \( G \cong C_p \times C_p \), and either \( V \) is trivial or \( p \) is odd, with \( V \) an irreducible \( \mathbb{C}[\operatorname{Aut}(G)] \)-module of dimension \( p+1 \), induced from the linear character \( (\phi, 1) \), with \(\phi\not= 1\), of the Borel subgroup of \( \operatorname{Aut}(G) \cong \operatorname{GL}_2(\mathbb{F}_p) \).
    \item[\textnormal{(d)}] \( p = 2 \), \( G \cong C_{2^k} \times C_2 \) (for \( k \geq 2 \)), and \( V \) is one of the \( 2^{k-2} \) irreducible \( \mathbb{C}[\operatorname{Aut}(G)] \)-modules of dimension 2.
    \item[\textnormal{(e)}] \( p \) is odd, \( G \cong C_{p^k} \times C_p \) (for \( k \geq 2 \)), and \( V \) is an irreducible \( \mathbb{C}[\operatorname{Aut}(G)] \)-module of dimension \( p \), as described in Section \ref{sec:restker32}.
\end{enumerate}
Moreover in all cases, the multiplicity of the composition factor is 1.
\end{theorem}
In part (c), we use the notation from \cite{Shapiro} for the irreducible characters of the Borel subgroup. See Section 6 for details.

Finally combining the above theorem with Bouc's semisimplicity result from \cite{Bouc0}, we obtain the complete list of composition factors of the $p$-biset functor of monomial Burnside rings.
\begin{corollary}
Let $G$ be a finite $p$-group and $V$ be a simple $\mathbb C[\out(G)]$-module. The simple biset functor $S_{G, V}$ is a composition factor of the $p$-biset functor $\mathbb CB^{\mathbb C^\times}$ if and only if either $(G, V)$ is determined by Theorem \ref{thm:main} or $G$ is cyclic and $V$ is primitive.  
\end{corollary}

Most of the paper is devoted to the proof of this theorem. To illustrate some of the composition factors, we include a few examples. For $p=2$, we have a list of composition factors with minimal group isomorphic to one of the groups $C_2, C_2\times C_2, C_4, C_8, C_4\times C_2$ or $C_8\times C_2$. Also for $p=3$, we have a similar list for groups $C_3, C_3\times C_3, C_9$ or $C_9\times C_3$.

This article is the first part of a two–paper series devoted to the structure of the monomial Burnside $p$-biset functor $B^{\mathbb{C}^\times}$. In the present work we determine the complete list of composition factors of the kernel $\mathbb S_p$ regarded as a $p$-biset functor. The forthcoming second part will be concerned with the description of the full lattice of subfunctors of $\mathbb CB^{\mathbb C^\times}$ and the evaluation of its simple summands at finite $p$-groups.

\section{Generalities on biset functors}\label{sec:pre}

This section introduces the foundational concepts of biset functors and establishes the tools needed for analyzing the kernel of the monomial linearization map. We refer to \cite{Bouc} and \cite{Bouc2} for details and to \cite{BCK} for results concerning composition factors. Although the results are more general we concentrate only on $p$-groups. 

\begin{nothing}{\bf Definition and Notation.}
    Let $\mathcal C_p:=\mathbb C\mathcal C_p$ be the $\mathbb C$-linear category of biset functors on the class of finite $p$-groups. In this category the set of morphisms from $G$ to $H$ is given by the Burnside group $\mathbb C B(H,G)$ of finite $(H\times G)$-bisets. A $\mathbb C$-linear functor $\mathcal C_p\to {}_\mathbb C\text{Vect}$ is called a \textit{$p$-biset functor (over $\mathbb C$)}. We denote the category of all $p$-biset functors by $\mathcal F_p$.

    For a finite $p$-group $G$, we write $\mathcal S'(G)$ for a set of representatives of isomorphism classes of finite $p$-groups of order strictly smaller than $|G|$. We also define 
    \[
    I_G:= \sum_{H\in\mathcal{S}'(G)} \mathbb CB(G,H)\circ \mathbb CB(H,G) 
    \]
    where $\mathbb CB(G,H)\circ \mathbb CB(H,G)$ is the subspace of $\mathbb CB(G, G)$ generated by all compositions $\alpha\circ\beta$ with $\alpha\in\mathbb CB(G,H)$ and $\beta\in\mathbb CB(H,G)$. Then one has
    \[
    \mathbb CB(G, G) \cong I_G \oplus \mathbb C[\out(G)].
    \]
    Here $\out(G)$ denotes the outer automorphism group of $G$. 
\end{nothing}

\begin{nothing}{\bf Simple functors.}
    Let $F$ be a $p$-biset functor. For any $p$-group $G$, the evaluation $F(G)$ is a module over the endomorphism ring $\mathcal E_G:={\rm End}_{\mathcal C_p}(G)$. Furthermore if $F$ is a simple functor, then $F(G)$ is a simple $\mathcal E_G$-module. The group $G$ is said to be \emph{minimal} for the simple functor $F$ if it is a group of minimal order with nonzero evaluation $F(G)$. It is known that minimal group for $F$ is unique up to isomorphism and the evaluation $F(G)$ is annihilated by the ideal $I_G$. In particular $F(G)$ is a simple $\mathbb C[\out(G)]$-module. Moreover by Theorem 4.3.10 of \cite{Bouc}, there is a bijective correspondence between the isomorphism classes of simple $p$-biset functors and the set of pairs $(G, V)$ where $G$ is a finite $p$-group and $V$ is a simple $\mathbb C[\out(G)]$-module, both taken up to isomorphims. We denote the simple functor corresponding to $(G, V)$ by $S_{G, V}$.
\end{nothing}

\begin{nothing}{\bf Restriction Kernels.}
    Let $F$ be a $p$-biset functor. The restriction kernel of $F$ at $G$, denoted $\mathcal{K}F(G)$, identifies elements in $F(G)$ that are annihilated by all maps $F(\alpha)$ for $\alpha\in \mathbb CB(H\times G)$ and $H\in\mathcal{S}'(G)$. Formally
    \[
    \mathcal KF(G) := \bigcap_{\substack{H\in\mathcal{S}'(G)\\ \alpha\in\mathbb CB(H\times G)}} \ker(F(\alpha): F(G)\to F(H))
    \]
    is called the \textit{restriction kernel} of $F$ at $G$. Note that the name is somehow misleading since we need to consider kernels of deflation maps as well as restrictions. It turns out that these modules plays a crucial role in determining composition factors of $F$. 
\end{nothing}
\begin{nothing}{\bf Composition factors.}
    By a \emph{composition factor} of $F$, we mean a simple biset functor $S$ which is isomorphic to a quotient $M/N$ of two subfunctors $N\subset M\subseteq F$. The \emph{multiplicity} $[F:S]$ of $S$ in $F$ is the maximal nonnegative integer $m$ such that there is a sequence 
    \[
    0\subseteq N_1\subset M_1\subseteq N_2\subset M_2 \subseteq \ldots N_m\subset M_m\subseteq F
    \]
    such that $M_i/N_i\cong S$. Then Theorem 2.5 of  \cite{BCK} gives a sufficient condition to reduce the calculation of composition factor multiplicities to those of restriction kernels. We include a complete proof of this theorem in Appendix D for the reader’s convenience.
\end{nothing}
\begin{nothing}
    We note that the version in \cite{BCK} works for modules over Green biset functors. Also as noted in \cite[Proposition 2.6]{BCK}, a consideration of a suitable cyclic submodule shows that the condition of the theorem is equivalent to the following:
\end{nothing}

\noindent\textbf{Condition A.} For any $x\in F(G)$, there exists $x'\in\mathcal{K}F(G)$ and $\alpha\in I_G$ such that $$x = x' + \alpha\cdot x.$$

\begin{remark}
    This theorem can be applied to the problem of determining the composition factors of $F$ in two steps.
    \begin{quote}
    \begin{enumerate}
        \item[Step 1.] Evaluate $\mathcal{K}F(G)$ for each $G$ and prove Condition A for $F$.
        \item[Step 2.] Determine the irreducible summands of the $\mathbb C[\out(G)]$-module $\mathcal{K}F(G)$, whenever it is nonzero.
    \end{enumerate}
    \end{quote}
    The second step depends on the structure of $\mathcal{K}F(G)$. For the first step we use Bouc's decompositions of $p$-biset functors given in \cite{Bouc2} to reduce the problem to simpler cases. We recollect necessary background without details. Following \cite{Bouc2}, for a finite $p$-group $G$, we write
        \[
\varphi_1^G =\frac{1}{|G|} \sum_{\substack{X\le G, M\unlhd G\\ M\le \Phi(G)\le X}} |X|\mu(X, G)\mu_{\unlhd G}(1, M)\ind\infl_{X/M}^G\defl\res^G_{X/M}.
        \]
        In the above sum, $\Phi(G)$ denotes the Frattini subgroup of $G$, $\mu(\cdot, \cdot)$ (resp. $\mu_{\unlhd G}(\cdot, \cdot)$) denotes the Mobius function 
        of the poset of subgroups (resp. of normal subgroups) of $G$. 
        By \cite[Theorem 4.8]{Bouc2}, this is an idempotent of the endomorphism ring $\mathbb CB(G\times G)$ and 
        by Corollary 3.7 of \cite{Bouc2} we have 
        \begin{itemize}
            \item $\res^G_H \varphi^G_1 = 0$ for any $H<G$ and
            \item $\defl^G_{G/M}\varphi^G_1 = 0$ for any $M\unlhd G$ with $M\cap \Phi(G) \not = 1$.
        \end{itemize}
        For a given $p$-biset functor $F$, we define
        \[
        \delta_\Phi F(G):= \varphi^G_1 F(G).
        \]
        Then by Proposition 5.2 of \cite{Bouc2}, the elements $u$ of $F(G)$ contained in $\delta_\Phi F(G)$ are characterized by the following two conditions:
        \begin{align}
    \res^G_H u &= 0\ \text{for all } H < G \text{ and} \label{eqn:delta1} \\
    \defl^G_{G/M} u & = 0 \text{ for all } M\unlhd G, M\cap \Phi(G) \not = 1. \label{eqn:delta2}
\end{align}
In particular we obtain 
\[
\mathcal{K}F(G)\subseteq \delta_\Phi F(G)
\]
and hence we may start calculations for $\mathcal{K}F(G)$ by determining $\delta_\Phi F(G)$. Moreover since any $x\in F(G)$ can be written as $x = \varphi^G_1x + (1-\varphi^G_1)x$ and since $(1-\varphi^G_1)\in I_G$, we also reduce the proof of the validity of Condition A to elements in $\delta_\Phi F(G)$.
\end{remark}
\begin{nothing}
    One may also consider the action of the idempotent $\tilde e_G^G$, as given in \cite{Bouc}. This is the image of the primitive idempotent $e_G^G$ of the Burnside ring $\mathbb CB(G)$ in $\mathbb CB(G\times G)$ as constructed in Section 2.12 of \cite{Bouc2}. By Lemma 2.18 of the same paper, for any biset functor $F$ and finite group $G$, we have
    \[
    \tilde e_G^GF(G) = \bigcap_{H<G} \ker(\res^G_H:F(G) \to F(H)).
    \] Hence we have
    \[
\mathcal{K}F(G)\subseteq \delta_\Phi F(G) \subseteq \tilde e_G^GF(G)
\]
\end{nothing}
\section{Monomial Linearization Map and its Kernel}\label{sec:mon}

This section introduces the monomial Burnside ring and defines the monomial linearization map, a key homomorphism between the monomial Burnside ring and the complex character ring. We also describe the kernel of this map by determining a primitive idempotent basis.

\begin{nothing}{\bf Definitions.}
Let $A$ be an abelian group and $X$ be an  $A\times G$-set. We call $X$ an {\it $A$-fibered $G$-set} if the $A$-action on $X$ is free with finitely many orbits. Together with $A\times G$-equivariant functions, the class of all $A$-fibered $G$-sets forms a category, written ${}_G\operatorname{set}^A$. Disjoint union of two $A$-fibered $G$-sets is an $A$-fibered $G$-set in an obvious way and it induces a categorical product on $\GsetA$. There is also a tensor product defined on $A$-fibered $G$-sets. Given $X, Y\in \GsetA$, the cartesian product $X\times Y$ becomes an $A$-set via 
\[ 
a\cdot (x, y) = (a\cdot x, a^{-1}\cdot y)
\] 
for $a\in A$ and $(x, y)\in X\times Y$. We denote the $A$-orbit containing $(x, y)$ by $x\otimes y$ and write $X\otimes Y$ for the set of all $A$-orbits in $X\times Y$. Then we make $X\otimes Y$ an $A\times G$-set via 
\[ 
(a, g)\cdot x\otimes y := agx\otimes gy.
\] 
It is straightforward to show that this is a categorical product on $\GsetA$. 

An $A$-fibered $G$-set $X$ is said to be \textit{transitive} if the group $G$ acts on the set of $A$-orbits transitively. Let $X\in\GsetA$ be transitive. Fix $x\in X$  and let $H$ be the stabilizer of the $A$-orbit $Ax$ of $x$ in $G$.  Then for any $h\in H$, there is an element $\phi(h)\in A$ such that $h\cdot x = \phi(h)\cdot x$ is satisfied. The function $\phi\colon H\to A$ is a homomorphism. We call the pair $(H, \phi)$ the \textit{stabilizing pair} of $x\in X$. If $y\in X$ is given, then the stabilizing pair of $y$ is $G$-conjugate to $(H, \phi)$.

Denote by $\mathcal M_G^A$ the the set of all such pairs. The group $G$ acts on $\mathcal M_G^A$ via conjugation: ${}^g(H, \phi) = ({}^gH, {}^g\phi)$. We write $[H, \phi]_G$ for the $G$-orbit containing $(H, \phi)$ and $\mathcal M_G^A/G$ for the set of $G$-orbits in $\mathcal M_G^A$. With this notation, the isomorphism classes of transitive sets in $\GsetA$ are in bijective correspondence with the set $\mathcal M_G^A/G$.
\end{nothing}
\begin{nothing}{\bf Monomial Burnside Rings.}
The \emph{monomial Burnside ring} \( B^{\mathbb{C}^\times}(G) \) of a finite group \( G \) is the Grothendieck ring of the category of \( \mathbb{C}^\times \)-fibered \( G \)-sets. As a \( \mathbb{C} \)-algebra, its basis is indexed by the set $\mathcal M^A_G/G$. The product can be identified as the linear extension of the following product on the basis elements. Let $[H, \phi]_G, [K, \psi]_G\in \mathcal M_G^A$. Then 
\[ 
[H,\phi]_G\cdot [K, \psi]_G = \sum_{x\in [H\backslash G/K]}[H\cap {}^xK, \phi|_{H\cap {}^xK}\cdot {}^x\psi|_{H\cap {}^xK}]_G 
\]

Fibered Burnside rings are studied in details in Dress \cite{Dress}, Barker \cite{Bar} and Boltje \cite{Bo8}. More recently in a joint work with Boltje \cite{BoCo}, the second author introduced fibered bisets and fibered biset functors. In this paper we are only interested in the case where $A = \mathbb C^\times$ is the group of nonzero complex numbers. In this case, we call a $\mathbb C^\times$-fibered $G$-set a \textit{monomial $G$-set} and write $\MonG = \mathbb C\otimes B^{\mathbb C^\times}(G)$. 
\end{nothing}

\begin{nothing} \textbf{Primitive Idempotents of $\MonG$.} 
The set of species of the ring $\MonG$ is determined by Dress \cite{Dress} and explicit descriptions of its primitive idempotents are determined independently by Barker \cite{Bar} and Boltje \cite{Bo8}. In this paper we follow Barker's description. 

Consider the set \[ \mathcal N(G) =\{ (H, hH')\colon H\le G, hH'\in H/H' \}. \]
We abbreviate $(H, hH')$ as $(H, h)$. Under the conjugation action of $G$ it becomes a $G$-set. We write $[H, h]_G$ for the $G$-orbit containing $(H, h)$. For each $(H, h)\in \mathcal N(G)$, we define
\[s_{H, h}^G: \MonG\to \mathbb C,\quad [K, \psi]_G\mapsto \sum_{\substack{xK\subset G\\ H\le {}^xK}}({}^x\psi)(h)\]

By Dress \cite{Dress} (also by \cite[Lemma 5.1]{Bar}), $s_{H, h}^G$ is a species of $\MonG$, any species of $\MonG$ is of this form and $s_{H, h}^G = s_{H', h'}^G$ if and only if $(H, h)$ is $G$-conjugate to $(H', h')$. Hence the set $\mathcal N(G)/G$ of $G$-orbits in $\mathcal N(G)$ is in bijection with the set of all species of $\MonG$. By the general theory, for each species $s_{H, h}^G$, there is a primitive idempotent $e_{H, h}^G$ such that 
\begin{align*}
 s_{H, h}^G(e_{K, k}^G) &= \begin{cases}
  1  &  (K, k) =_G (H, h) \\
   0 & \text{otherwise.}
  \end{cases} \\
\end{align*}
The following formula for primitive idempotents can found in \cite[Theorem 5.2]{Bar}.
\end{nothing}

\begin{theorem}[Barker]
For each $(H, h)\in \mathcal N(G)$, we have 
\[ e_{H, h}^G = \frac{1}{|N_G(H, h)|}\sum_{[K, \psi]_G\in\mathcal M_G/G}|K|\mu_G(K, \psi; H, h) [K, \psi]_G\]         
\end{theorem}
\noindent
Here $N_G(H, h)$ denotes the stabilizer of $(H, h)$ in $G$ and $\mu_G(K, \psi; H, h)$ is the monomial Mobius function defined in \cite[Section 5]{Bar}. We do not need the explicit definition. Note that we have $|\mathcal M_G/G| = |\mathcal N(G)/G|$. In \cite{BoYi}, there is another description of primitive idempotents of $\MonG$ parameterized by the set $\mathcal M_G/G$. 

\begin{nothing} \textbf{Monomial linearization map.}
We denote by $\mathcal R_\mathbb C(G)$ the complex character ring of $G$. As an abelian group, it is free on the set $\operatorname{Irr}(G)$ of irreducible characters of $G$. The product is given by the product of characters. The monomial linearization map translates the combinatorial structure of monomial $G$-sets into the representation-theoretic framework of $\mathbb CG$-modules by associating $(H, \phi)$ to the induced character $\ind_H^G\phi$. More precisely, the \emph{monomial linearization map}
\[
\ling \colon B^{\mathbb{C}^\times}(G) \to \mathcal{R}_{\mathbb{C}}(G)
\]
is defined as the linear extension of the map
\[
\ling([H, \varphi]_G) = \operatorname{Ind}_H^G \varphi.
\]
This map is surjective by Brauer's Induction Theorem. In this paper we are interested in the kernel of $\ling$ over $\mathbb C$. It is possible to define $\ling$ as a functor ${}_G\text{set}^{\mathbb C^\times} \to {}_{\mathbb CG}\text{mod}$ by mapping a monomial $G$-set $X$ to the monomial $\mathbb CG$-module $\mathbb CX$. Then $\ling$ is the map induced by this functor on the Grothendieck rings of these categories.

The following proposition evaluates images of primitive idempotents under $\ling$ in terms of primitive idempotents of $\mathbb C\mathcal{R}_\mathbb C (G)$. The result can be found in \cite{Bpre}, we include a proof for the reader's convenience.
\end{nothing}

\begin{proposition}[Boltje]
    Let $[H, h]_G\in\mathcal N(G)/G$. Then
    \begin{equation*}
    \ling(e_{H, h}^G) = \begin{cases}
              e_h^G & \text{if } H = \langle h\rangle,\\
               0 & \text{otherwise}.
          \end{cases}
\end{equation*}
\end{proposition}
\begin{proof}
   For each $g\in G$, we need to evaluate the $e_g^G$-coordinate of $\ling(e_{H, h}^G)$. This coordinate is is given by the composition 
   \[
   s_g^G\circ \ling: \MonG \to \mathbb C\mathcal R_\mathbb C(G)\to \mathbb C
   \]
   where $s_g^G(f) = f(g)$ for any $f\in \mathbb C\mathcal R_\mathbb C(G)$. The composite function at the basis element $[K, \psi]_G\in \mathcal M_G/G$ is evaluated as follows:
   \begin{eqnarray*}
       s_g^G\circ \ling([K, \psi]_G) &=& s_g^G(\ind_K^G\psi) = \ind_K^G\psi(g)\\
       &=& \sum_{xK\subseteq G:g\in {}^xK}({}^x\psi) (g)\\
       &=& s_{\langle g\rangle, g}^G([K, \psi]_G).
   \end{eqnarray*}
   Now it follows that $s_g^G\circ \ling = s_{\langle g\rangle, g}^G$ and hence
   \[
   s_g^G\circ \ling(e_{H, h}^G) = s_{\langle g\rangle, g}^G(e_{H, h}^G) = \begin{cases}
              e_h^G & \text{if } H = \langle h\rangle =_G \langle g\rangle,\\
               0 & \text{otherwise}.
          \end{cases} 
   \]
\end{proof}
\begin{corollary}\label{cor:ker-ipot}
The kernel of \( \ling \) is spanned by primitive idempotents \( e_{H,h}^G \) corresponding to pairs \( (H, h) \) where \( H \neq \langle h \rangle \). Explicitly,
\[
\ker(\ling) = \operatorname{span}_{\mathbb C}\{ e_{H,h}^G \mid (H,h)\in \mathcal N(G), H \neq \langle h \rangle \}.
\]
\end{corollary}

As mentioned in the introduction, when regarded as a morphism of fibered $p$-biset functors, the kernel of the linearization map is simple, isomorphic to $\scp$. From now on we write $\scp$ for this kernel regarded as a (fibered) $p$-biset functor. 
Combining with Corollary \ref{cor:ker-ipot}, we conclude the following result.

\begin{corollary}
    For any finite $p$-group $G$, the  vector space $\scp(G)$ has a basis consisting of the primitive idempotents $e_{H, h}^G$ of $\MonG$ as $(H, h)\in \mathcal N(G)$ runs over the $G$-orbits of the pairs $(H, h)$ where $H\not = \langle h\rangle$. 
\end{corollary}

\begin{nothing}
    We also include biset actions on primitive idempotents. Details of formulas can be found in \cite{Bar}, \cite{BoYi} and \cite{CoYi}.
\end{nothing}
\begin{theorem}[\cite{Bar}, \cite{BoYi}, \cite{CoYi}]\label{thm:bisetactions} 
    Let $K\le G$ and $N\unlhd G$. 
    
    \textbf{(a)} For any $(H,h)\in\mathcal{N}(G)$, we have
    \[
    \res^G_K e_{H, h}^G = \sum_{(H',h')} e_{H', h'}^K
    \]
    where the sum is over all pairs $(H', h')\in\mathcal{N}(K)$ up to $K$-conjugacy which are $G$-conjugate to $(H,h)$.

    \smallskip

    \textbf{(b)} For any $(H, h)\in\mathcal{N}(K)$, we have
    \[
    \ind_K^G e_{H, h}^G = \frac{|N_G(H,h)|}{|N_K(H,h)|}e_{H,h}^G
    \]
    \smallskip
  
    \textbf{(c)} For any $(H/N, hN)\in\mathcal{N}(G/N)$, we have
    \[
    \infl_{G/N}^G e_{H/N,hN}^{G/N} = \sum_{L, l}e_{L, l}^G
    \]
    where the sum is over all pairs $(L,l)\in\mathcal{N}(G)$ such that $(LN/N, lN)$ is $G/N$-conjugate to $(H/N, hN)$.

    \smallskip

    \textbf{(d)} For any $(H, h)\in\mathcal{N}(G)$, we have
    \[
    \defl^G_{G/N} e^G_{H, h} = m_{H, h}^{G, N} e^{G/N}_{HN/N, hN}.
    \]
    for some constant $m_{H, h}^{G, N}$, called the deflation number for $(H,h)$ at $N$. 

    \smallskip

    \textbf{(e)} The deflation number $m_{G, g}^N:= m_{G, g}^{G, N}$ is given by
    \[
    m_{G, g}^N = \frac{1}{|NG'|}\sum_{\substack{V\le G\\ VN = G}} |V\cap gG'|\mu(V, G).
    \]
    Here $G'$ denotes the derived subgroup of $G$.

    \smallskip
    
    \textbf{(f)} Let $\lambda:G' \to G$ be an isomorphism. Then
    \[
    \iso_{G', G}^\lambda (e_{H, h}^G) = e_{\lambda(H), \lambda(h)}^{G'}.
    \]
    
\end{theorem}
\begin{nothing}
    In the rest of the paper we mostly need explicit formulas for the deflation numbers for pairs of the form $(G, g)$ where $G$ is an abelian group. The only exception is the case where $N=G'$ is the derived subgroup of $G$, which is covered below. We also recall the formulas for elementary abelian $p$-groups.     
\end{nothing}
\begin{lemma}\label{lem:defnoderived}
    Let $G$ be a $p$-group and $G'$ its derived subgroup. Then $m_{G, g}^{G'} = 1$.
\end{lemma}
\begin{proof}
    If $G'=1$, then the deflation number is clearly 1. Hence assume $G'\not = 1$. Since $G$ is a $p$-group, we have $G'\le \Phi(G)$. Hence for any subgroup $V\le G$, the equality $VG' = G$ implies $V = G$. Hence the sum in Theorem \ref{thm:bisetactions} (e) reduces to only one term which is 1, as required.
\end{proof}
\begin{lemma}\label{lem:defnoFrat}
    Let $G$ be an abelian $p$-group. Then for any subgroup $1\not = N\le G$, we have
    \[
    m_{G, g}^N =\frac{1}{|N\cap \Phi(G)|} m_{G/\Phi(G), g\Phi(G)}^{N\Phi(G)/\Phi(G)}.
    \]
\end{lemma}
\begin{proof}
    By Theorem \ref{thm:bisetactions} (e), the deflation number on the left hand side is given by
    \begin{eqnarray*}
m_{G, g}^{N} &=& \frac{1}{|N|}\sum_{\substack{V\le G\\ VN = G}} |V\cap \{g\}|\mu(V, G)\\
    &=& \frac{1}{|N|}\sum_{\substack{V\le G\\ VN = G, g\in V}} \mu(V, G)        
    \end{eqnarray*}
Similarly, the deflation number on the right hand side is
  \begin{eqnarray*}
    m_{G/\Phi(G), g\Phi(G)}^{N\Phi(G)/\Phi(G)}  &=& \frac{1}{|N\Phi(G)/\Phi(G)|}\sum_{\substack{\bar V\le \bar G\\ VN\Phi(G) = G,\, g\Phi(G)\in \bar V}} \mu(\bar V, \bar G)
    \end{eqnarray*}
 where $\bar V = V/\Phi(G)$ and $\bar G= G/\Phi(G)$.
    The result follows by comparing the two expressions.
\end{proof}
\begin{nothing}
    Finally we recall the evaluation of the deflation numbers for elementary abelian $p$-groups. Details can be found in \cite{CoYi}.
\end{nothing}
\begin{lemma}[\cite{CoYi}]\label{lem:defnoelem}
    Let $G$ be an elementary abelian $p$-group of rank $k$ and let $T\le G$ be a subgroup of order $p$. Also let $g\in G$. Then
    \[
    m_{G, g}^T = \begin{cases}
              \frac{1-p^{k-1}}{p} & \text{if } g = 1,\\
              \frac{1}{p} & \text{if } 1\not = g\in T\\
            \frac{1 - p^{k-2}}{p} & \text{if } g\not\in T.
          \end{cases} 
    \]
    
\end{lemma}

\section{Reduction to Restriction Kernel}\label{sec:reduc1}

In this section, we reduce the problem of determining the composition factors of \( \kerlin \), viewed as a \( p \)-biset functor, to analyzing the restriction kernels \( \mathcal{K}\kerlin(G) \). This is done in two steps. First we prove that Condition A holds for the functor $\scp$ and hence its composition factors are parameterized by the composition factors of restriction kernels of $\scp$. This is done in this section. Then in the rest of the paper, we evaluate restriction kernels explicitly. 
\begin{theorem}\label{thm:red2Rest}
    Let $G$ be a finite $p$-group which is non-trivial and not isomorphic to $C_p\times C_p$. Also let $x\in \kerlin(G)$. Then \[ x = x' + \alpha x\] for some $x'\in \mathcal K\kerlin(G)$ and some $\alpha\in I_G$.
\end{theorem} 
\begin{nothing}{\bf Outline of the Proof of Theorem \ref{thm:red2Rest}.} Let $G$ be a finite $p$-group. If it satisfies the theorem, we say that $G$ {\emph{satisfies condition (C)}}. We divide the proof of Theorem~\ref{thm:red2Rest} into several steps:
\begin{enumerate}
    \item \textbf{Reduction to abelian \( p \)-groups (Proposition~\ref{pro:nonAb}):} We first show that it suffices to consider abelian $p$-groups.
    \item \textbf{Analysis of groups with no \( C_p \)-factor (Proposition~\ref{pro:noCp}):} We prove that for such groups, the restriction kernel \( \mathcal{K}\kerlin(G) \) is zero unless \( G \) is cyclic.
    \item \textbf{Remaining cases:} We finally consider
    \begin{enumerate}
        \item groups having at least 2 cyclic factors of order $\ge p^2$ (Proposition \ref{pro:2Cpk}),
        \item groups having rank $\ge 3$ and at most one cyclic factor of order $\ge p^2$ (Proposition \ref{pro:1Cpk2Cp}),
        \item groups isomorphic to $C_{p^m}\times C_p$ with $m\ge 2$ (Proposition \ref{pro:1Cp1Cpk}) and
        \item finally the group $C_p\times C_p$ (Corollary \ref{cor:resker2}).
    \end{enumerate}   
\end{enumerate}
\end{nothing}
\begin{nothing}
The first step of the proof is to reduce the proof to elements satisfying $x = \varphi_1^G\cdot x$, that is, for $x\in\delta_\Phi\kerlin(G)$. This is done by Remark 2.6. To be more precise, we suppose $x\in \kerlin(G)$. By Corollary \ref{cor:ker-ipot}, we may write $x$ as \[ x = \sum_{(H, h): H\not = \langle h\rangle} \alpha_{H, h}e_{H, h}^G \] 
where the coefficients are determined by the equation $\alpha_{H, h}e_{H, h}^G = e_{H, h}^G\cdot x$.
On the other hand, we can write $$x = \varphi_1^G\cdot x + (1 - \varphi_1^G)\cdot x.$$ In this decomposition, $(1 - \varphi_1)\cdot x\in \mathcal I\kerlin (G)$.
Notice that $1-\varphi_1^G\in I_G$, so by putting $$y = \varphi_1^G\cdot x =x - (1 - \varphi_1)\cdot x$$ it is sufficient to prove the theorem for $y$ instead of $x$. Note also that $y = \varphi_1^G\cdot y$ and by \cite[Proposition 5.2]{Bouc2}, the element $y$ satisfies Equation (\ref{eqn:delta1}) and Equation (\ref{eqn:delta2}).
Furthermore, the element $y$ is contained in the span of the top idempotents, that is, idempotents of the form $e_{G, g}^G$ for some $g\in G$, that is, the coordinate-decomposition of $y$ in the primitive idempotent basis is given by
\begin{equation}\label{eqn:forY}
y = \sum_{(G, g): G\not = \langle g\rangle} \alpha_{g}e_{G, g}^G. 
\end{equation}
This is because restriction of $y$ to any proper subgroup of $G$ must be zero. As above, we have $\alpha_ge_{G, g}^G = e_{G, g}^G\cdot y$.
\end{nothing} 

\begin{proposition}\label{pro:nonAb}
    Let $G$ be a non-abelian $p$-group. Then $\delta_\Phi\kerlin(G) = 0$. In particular, $\mathcal K\kerlin(G)$ also equals to $0$ and $G$ satisfies condition (C). 
\end{proposition}
\begin{proof} Let $y\in\delta_\Phi\kerlin(G)$ be as in Equation (\ref{eqn:forY}).
    Since $\{1\} \not =G'\subseteq \Phi(G)$, by Equation (\ref{eqn:delta2}), we must have $\defl^G_{G/G'} y = 0$. On the other hand, for any $gG'\in G/G'$ we have
\[
\defl^G_{G/G'} e_{G, g}^G = e_{G/G', gG'}^{G/G'}
\]
since in this case $m_{G, g}^{G'} = 1$ by Lemma \ref{lem:defnoderived}. 
Moreover for any $g,h\in G$, we have
\[
e_{G, g}^G = e_{G, h}^G \text{ if and only if } e_{G/G', gG'}^{G/G'} = e_{G/G', hG'}^{G/G'}
\]
Therefore deflation of $y$ to $G/G'$ is equal to
\[
\defl^G_{G/G'} y = \sum_{(G/G', gG')} \alpha_{g}e_{G/G', gG'}^{G/G'} 
\]
Now it is clear that the above sum is 0 if and only if all the coefficients $\alpha_g$ are zero, that is, if and only if $y = 0$. Thus $\delta_\Phi\kerlin(G) = 0$, as required. The rest follows easily since $\mathcal{K}\kerlin(G)\subseteq \delta_\Phi\kerlin(G)$ \end{proof} 
\begin{nothing} For the rest of the proof, we assume that $G$ is an abelian $p$-group. Using the classification of finite abelian groups, we write and fix a decomposition of $G$ into a product of cyclic groups as follows. 
\begin{equation}\label{eqn:decompG}
G = C_{p^{r_1}}\times C_{p^{r_2}}\times \ldots\times C_{p^{r_k}}
\end{equation}
where the integers $r_i$ are chosen so that $1\le r_1\le r_2\le \ldots\le r_k$. 
\end{nothing}
\begin{proposition}\label{pro:noCp}
    Let $G$ be as in Equation (\ref{eqn:decompG}). Suppose $G$ has no direct factor isomorphic to $C_p$. Then

    \smallskip

    \noindent \textup{(i)} $G$ satisfies condition (C),

    \smallskip
    
    \noindent \textup{(ii)} $\mathcal K\kerlin(G) = 0$ if and only if $G$ is non-cyclic.
\end{proposition}
\begin{proof}
 Since $G$ has no $C_p$-factor, the Frattini subgroup $\Phi(G)$ of $G$ can be written as a product of maximal subgroups of its direct factors, that is,
$$\Phi(G) \cong C_{p^{r_1-1}}\times C_{p^{r_2-1}}\times \ldots\times C_{p^{r_k-1}}.$$
It is clear from this decomposition that
it intersects all non-trivial subgroups of $G$ non-trivially. In particular, by Equation (\ref{eqn:delta2}) above, we conclude that for any $y\in\delta_\Phi\kerlin(G)$, we have $\defl_{G/H}^G y = 0$ for any $1\not = H\le G$ and hence $y\in \mathcal K\kerlin(G)$. In particular $\mathcal{K}\kerlin(G) = \delta_\Phi\kerlin(G)$ and $G$ satisfies condition (C). 

Next we prove that $\mathcal K\kerlin(G) = 0$ if $k$, the number of cyclic factors, is at least two. Hence suppose $G$ is non-cyclic and let $T=\langle t \rangle$ be a minimal subgroup of $G$. Let $y\in\mathcal K\kerlin(G)$ be as in Equation (\ref{eqn:forY}). For $g\in G$ consider the sum
\[
y_{T, g}:= \sum_{i=0}^{p-1} \alpha_{gt^i}e_{G, gt^i}^G. 
\] It is clear that 
\begin{enumerate}
    \item[(a)] $\defl_{G/T}^G y_{T, g} = m\cdot e_{G/T, gT}^{G/T}$ for some scalar $m$ and
    \item[(b)] $e_{G/T, gT}^{G/T}\cdot \defl_{G/T}^G (y -y_{T, g}) = 0$
\end{enumerate}
In other words, 
\[
\defl^G_{G/T} y = m\cdot e_{G/T, gT}^{G/T}  + \text{ terms not containing } e_{G/T, gT}^{G/T}
\]
for some scalar $m$. Hence, since $\defl_{G/T}^G y =0$, we must have that $m=0$. We calculate $m$ as follows. By the explicit sum in Theorem \ref{thm:bisetactions} (e), we have $m_{G, gt^i}^T = 1/p$ and hence \[ \defl_{G/T}^G e_{G, gt^i}^G = \frac{1}{p}e_{G/T, gT}^{G/T}. \]
Therefore $m = (\sum_{i=0}^{p-1}\alpha_{gt^i})/p$. Since $m$ must be zero, we obtain the following equation:
\[
\alpha_g + \alpha_{gt} + \ldots \alpha_{gt^{p-1}} =0
\]
or equivalently 
\begin{equation}\label{eqn:alphag}
    -\alpha_g = \alpha_{gt} + \ldots \alpha_{gt^{p-1}} = \sum_{j=1}^{p-1}\alpha_{gt^j}
\end{equation}

Let $u\in G$ be another element of order $p$, the above calculations give
\[
\sum_{i=0}^{p-1} \alpha_{gu^jt^i} = 0
\]
for each $j=1, \ldots, p-1$. Summing these equations over $j$, we get
\[
\sum_{j=1}^{p-1}\sum_{i=0}^{p-1} \alpha_{gu^jt^i} = 0.
\]
Since the set
\[
\{ u^jt^i\mid i=0, 1, \ldots, p-1; j=1, 2,\ldots, p-1 \}\]
is equal to the set 
\[\{ (ut^i)^j \mid i=0, 1, \ldots, p-1; j=1, 2,\ldots, p-1 \} 
\]
we may rearrange the sum as
\[
\sum_{i=0}^{p-1}\Big( \sum_{j=1}^{p-1} \alpha_{g(ut^i)^j}\Big) = 0.
\]
By Equation (\ref{eqn:alphag}), the inner sum is independent of $i$ and is equal to $-\alpha_g$. Thus we get $p\cdot \alpha_g = 0$.
Hence  for each $g\in G$, $\alpha_g =0$, that is, $y = 0$. In particular, $\mathcal{K}\kerlin(G) = 0$, as required. Finally if $G$ is cyclic, then $\mathcal K\kerlin(G)\neq 0$, as we will see in Section \ref{sec:restker}.
\end{proof}

\begin{proposition}\label{pro:2Cpk}
     Let $G$ be as in Equation (\ref{eqn:decompG}). If $G$ has at least 2 direct factors of order $\ge p^2$, then $\mathcal K\kerlin(G) = 0$. In particular $G$ satisfies condition (C). 
     \end{proposition}

\begin{proof}
    Let $T, U\le \Phi(G)$ be minimal subgroups of $G$ such that $T\cap U =1$. Repeating the argument in the previous proof with these choices of $T$ and $U$, we arrive at the same conclusion that $\mathcal K\kerlin(G) = 0$. With this at hand, the second claim holds trivially. 
\end{proof}
\begin{proposition}\label{pro:1Cpk2Cp}
    Let $G$ be as in Equation (\ref{eqn:decompG}). Assume that $k\ge 3$ and also that $G$ has at most one cyclic factor of order at least $p^2$. Then $\mathcal K\kerlin(G) = 0$. In particular $G$ satisfies condition (C).
\end{proposition}
\begin{proof} To begin with, notice that the number of $C_p$-factors of $G$ is at least 2 and that makes the proof different from the above two cases since the Frattini subgroup is either trivial or has a unique subgroup of order $p$. Let $T, U$ be subgroups of $G$ of order $p$ such that $(T\times U)\cap \Phi(G) = 1$. We want to repeat the argument in Proposition \ref{pro:noCp} with these choices of $T$ and $U$. The only difference is in the deflation numbers. Let $g\in G$ be an element such that $\alpha_g\not = 0$. We divide the calculation into four cases:
\begin{enumerate}
    \item[] \textbf{Case (1):} Let $g\in \Phi(G)$. Then $gx\in\langle x\rangle\Phi(G)$ for any $x\in T\times U$. Hence by Lemma \ref{lem:defnoFrat} and Lemma \ref{lem:defnoelem}, we obtain
    \[
    p\cdot m_{G, h}^{\langle x\rangle} = \begin{cases}
              1 & \text{if } h = gx^{i}, 1\le i\le p-1\\
            1 - p^{k-1} & \text{if } h = g.
          \end{cases} 
    \]
    In particular Equation (\ref{eqn:alphag}) becomes
    \begin{equation}\label{eqn:alphag2}
        (p^{k-1}-1)\alpha_g = \alpha_{gx} + \ldots \alpha_{gx^{p-1}} 
    \end{equation}
    The rest of the previous argument applies and we obtain the equality $(p^{k-1}-1)\alpha_g = 0$. Once more, since $\alpha_g$ is nonzero and additionally $k\ge 3$, we arrived at a contradiction. Hence $\alpha_g = 0$ which also implies that $y = 0$, completing the proof of this case. 
    \item[] \textbf{Case (2):} If $|\langle g\rangle| = p^{r_k}$, then $g\not \in\Phi(G)$ and $gx\not\in\Phi(G)$ for any $x\in T\times U$. Therefore $m_{G, gx}^{\langle x\rangle}=1/p$ and the argument in Proposition \ref{pro:noCp} applies directly.
    
    \item[] \textbf{Case (3):} Finally if $g\not\in\Phi(G)$ and $o(g) = p$, then we may choose $T=\langle g\rangle$ so that $\langle g\rangle \le T\times U$. In this case, again by by Lemma \ref{lem:defnoFrat} and Lemma \ref{lem:defnoelem}, the deflation numbers are given by
        \[
    p\cdot m_{G, gx}^{\langle x\rangle} = \begin{cases}
              1 & \text{if } x\not\in T,\\
            1 - p^{k-1} & \text{if } x \in T.
          \end{cases} 
    \]
    Repeating the same argument once more, we get
    \[
    (p(1-p^{k-1} -1) +1)\alpha_g = 0
    \]
    and hence get a contradiction since $\alpha_g\not = 0$.
\item[] \textbf{Case (4):} There is a final case where $g\not\in \Phi(G)$ and $\langle g\rangle$ is not a maximal subgroup. Then the deflation numbers are equal to $1/p$ and calculations similar to the ones above shows that idempotents $e_{G, g}^G$ for such elements also satisfy the assertion. We leave the necessary modifications to the reader.
\end{enumerate}
\end{proof}

Now it remains to prove the assetion for groups of the form $C_{p^r}\times C_p$ for an integer $r\ge 1$.
\begin{proposition}\label{pro:1Cp1Cpk}
    The group $G = C_{p^r}\times C_p$ with $r>1$ satisfies condition (C).
\end{proposition}

\begin{proof}
 We need to prove that $y = \varphi_1^G\cdot y\in\partial\kerlin(G)$ satisfies the assertion of the theorem. Since $y$ is in the span of the primitive idempotents $e_{G, g}^G, g\in G$, it is sufficient to prove the claim for these idempotents. 

 First we evaluate the product $E_g= E_g^G:=\varphi_1^G\cdot e_{G, g}^G$ more explicitly. By \cite[Proposition 3.6]{Bouc}, 
\begin{eqnarray*}
E_g &=& \frac{1}{|G|}\sum_{\substack{X\le G, M\unlhd G\\ M\le \Phi(G) \le X\le G}} |X| \mu(X, G)\mu_{\unlhd G}(1, M)\big(\Upsilon\cdot e_{G,g}^G\big)\\
&=& \frac{1}{|G|}\sum_{\substack{M\le G\\ M\le \Phi(G)}} |G|\mu(1, M)\infl_{G/M}^G\defl^G_{G/M} \cdot\ e_{G, g}^G
\end{eqnarray*}
where $\Upsilon =\ind\infl_{X/M}^G\defl\res^G_{X/M}
$.
The above equality holds since restriction of $e_{G, g}^G$ to any proper subgroup is zero. We also sum over all subgroups since $G$ is abelian. Moreover if $M$ is not elementary abelian, we have $\mu(1, M) = 0$. Thus, since the Frattini subgroup is cyclic, there are only two terms in the above sum, namely for the trivial subgroup $M = 1$ and for the unique subgroup $Z$ of order $p$ of $\Phi(G)$. Hence we obtain
\begin{equation}
E_g  = (\infl^G_G\defl^G_G - \infl_{G/Z}^G\defl^G_{G/Z}) \cdot e_{G, g}^G
\end{equation}
The first biset is the identity biset. We evaluate the action of the second term. Since $Z = \langle z \rangle$ is contained in the Frattini subgroup, the deflation number $m_{G, g}^Z$ is equal to $1/p$. Moreover by Theorem \ref{thm:bisetactions}, we have
\[
\infl_{G/Z}^G e_{G/Z, gZ}^{G/Z} = \sum_{i=0}^{p-1}e_{G, gz^i}^G.
\]
Combining all these, we get
\begin{equation}\label{eqn:phieee}
E_g=  e_{G, g}^G - \frac{1}{p}\sum_{i=0}^{p-1}e_{G, gz^i}^G.
\end{equation}
For further use, we list subgroups of order $p$ in $G$. 
Let $T_0=\langle t\rangle\not= Z$ be a subgroup of order $p$. Then for each $j= 1, \ldots, p-1$, the element $t_j:= tz^j$ generates a subgroup of order $p$, say $T_j$. Moreover for $i\not = j$, we clearly have $T_i \cap T_j = \{1\}$. Hence $Z, T_0, T_1, \ldots, T_{p-1}$ is the complete list of subgroups of order $p$.

Now by the above calculations, we already have that $\defl^G_{G/Z}E_g = 0$. Hence we only need to consider deflation to the $T_j$. There are three cases to consider.
\begin{enumerate}
    \item[] \textbf{Case (1):} Suppose that $\langle g\rangle$ is a maximal cyclic subgroup of $G$. Then for any $j= 0, \ldots, p-1$, the deflation number $m_{G, g}^{T_j}$ is zero. Indeed in this case there is a unique maximal subgroup containing $g$, namely $\langle g\rangle$ and hence we get
    \begin{eqnarray*}
    m_{G, g}^{T_j} &=& \frac{1}{p}\sum_{\substack{V\le G\\ VT_j = G, g\in V}} \mu(V, G) \\
    &=& \frac{1}{p} (-1 + 1) = 0.
    \end{eqnarray*}
    Hence in this case, $E_g$ lies in $\mathcal K\kerlin(G)$. This completes the proof of the current case.
    \item[] \textbf{Case (2):} Suppose $g\in\Phi(G)$. Then for each $j=0, 1, \ldots, p-1$, we have
    \[
    m_{G , g}^{T_j} = \frac{1-p}{p}
    \]
    Indeed, since $g$ is contained in any maximal subgroup, each maximal subgroup $V$ contributes $-1/p$ to the sum. Also the condition $U\cdot T_j = G$ is satisfied only if $V$ is a maximal cyclic subgroup. Together with the contribution $1$ from $V = G$, the result is obtained. 

    Now fix $j\in\{0, 1, \ldots, p-1\}$. We evaluate $j_{T_j}^G\cdot E_g$ where $j_N^G = \infl_{G/N}^G\defl^G_{G/N}$. Notice that for any $x\in G$, we have
    \[
    j_{T_j}^G\cdot e_{G, x}^G = \sum_{i=0}^{p-1} e_{G, xt_j^i}^G.
    \]
    Hence, using linearity and Equation (\ref{eqn:phieee}), we obtain
    {\allowdisplaybreaks
    \begin{eqnarray*}
    j_{T_j}^G\cdot E_g &=& j_{T_j}^G\cdot e_{G, g}^G - \frac{1}{p}\sum_{i=0}^{p-1}j_{T_j}^G\cdot e_{G, gz^i}^G\\ 
    &=& \sum_{k=0}^{p-1} e_{G, gt_j^k}^G - \frac{1}{p}\sum_{i=0}^{p-1}\sum_{k=0}^{p-1} e_{G, gz^it_j^k}^G \\
    &=& \sum_{k=0}^{p-1} e_{G, gt_j^k}^G - \frac{1}{p}\sum_{k=0}^{p-1}\sum_{i=0}^{p-1} e_{G, gz^it_j^k}^G \\
    &=& \sum_{k=0}^{p-1} \Big( e_{G, gt_j^k}^G - \frac{1}{p}\sum_{i=0}^{p-1} e_{G, gz^it_j^k}^G\Big) \\
    &=& \sum_{k=0}^{p-1} \varphi_1^G\cdot e_{G, gt_j^k}^G\\
    &=& \varphi_1^G\cdot e_{G, g}^G  + \sum_{k=1}^{p-1} \varphi_1^G\cdot e_{G, gt_j^k}^G
    \end{eqnarray*}}
    Now summing over all $j$, one gets
    \begin{equation}\label{eqn:phieeefinal}
    \sum_{j=0}^{p-1}j_{T_j}^G\cdot E_g = p\cdot E_g + \sum_{j=0}^{p-1}\sum_{k=1}^{p-1} \varphi_1^G\cdot e_{G, gt_j^k}
    \end{equation}
    We claim that the last term in the above equation is zero. Indeed, we have
    \begin{eqnarray*}
        \sum_{j=0}^{p-1}\sum_{k=1}^{p-1} \varphi_1^G\cdot e_{G, gt_j^k}^G &=& \sum_{j=0}^{p-1}\sum_{k=1}^{p-1} \Big( e_{G, gt_j^k}^G - \frac{1}{p}\sum_{i=0}^{p-1} e_{G, gt_j^kz^i}^G\Big)\\
        &=& \sum_{j=0}^{p-1}\sum_{k=1}^{p-1} e_{G, gt_j^k}^G - \frac{1}{p}\sum_{j=0}^{p-1}\sum_{k=1}^{p-1} \sum_{i=0}^{p-1} e_{G, gt_j^kz^i}^G
    \end{eqnarray*}
    It is clear that the set $\{t_j^k \mid k=1,\ldots, p-1, j=0,\ldots, p-1\}$ has no repetition and is equal to the set of all elements of $G$ of order $p$ not contained in $Z$. On the other hand $\{t_j^kz^i \mid k=1,\ldots, p-1, j=0,\ldots, p-1, i=0,\ldots, p-1\}$ is also equal to the same set but this time each element appears $p$ times. Thus each idempotent $e_{G, gt_j^k}^G$ appears in the right hand side $1 - (1/p)\cdot p = 0$ times. Hence the whole sum is zero. Substituting in Equation (\ref{eqn:phieeefinal}), we obtain
    \begin{equation}
    E_g = \varphi_1^G\cdot e_{G, g}^G = \frac{1}{p}\cdot \sum_{j=0}^{p-1}j_{T_j}^G\cdot E_g
    \end{equation}
    In particular, the element $E_g$ satisfies the conclusion of the theorem, as required.
\item[] \textbf{Case (3):} There is a final case where $g\not\in \Phi(G)$ and $\langle g\rangle$ is not a maximal subgroup. This case is very similar to Case (4) of the previous theorem and once more we leave details to the reader.
\end{enumerate}
\end{proof}

\noindent As a corollary of the theorem, we obtain the following partial parameterization of simple composition factors of $\kerlin$.
\begin{corollary}\label{cor:resker1}
    Let $S$ be a simple biset functor with minimal group $H\not \cong C_p\times C_p$ and put $V = S(H)$. Then the multiplicity of $S$ as a composition factor of $\kerlin$ as a biset functor is equal to the multiplicity of $V$ as a composition factor of $\mathcal K\kerlin(H)$ as a $\mathbb C[\out(H)]$-module.
\end{corollary}

\begin{nothing} We need to study the case $G= C_p\times C_p$ separately. First note that the above proof does not apply to this case. Although it could have been covered in Proposition \ref{pro:1Cp1Cpk} since the idempotent $\varphi_1^G$ acts on the primitive idempotents $e_{G, g}^G$ as the identity, the argument cannot be applied. Actually the conclusion of the theorem does not hold for $C_p\times C_p$. Nevertheless we still claim that the composition factors in this case are also the ones coming from the restriction kernel.
\end{nothing}
Let $\mathcal D$ be the full subcategory of the biset category $\mathcal C$ on the groups (isomorphic to) $1, C_p$ and $E_2:= C_p\times C_p$. Then for any simple $\mathbb C[\out(E_2)]$-module $V$, we have
\[
[\kerlin : S_{E_2, V}^\mathcal{C}] = [\res^\mathcal{C}_\mathcal{D}\kerlin : S_{E_2, V}^\mathcal{D}]
\]
Note that $\kerlin(1) = 0$ and $\dim_\mathbb C\kerlin(C_p) = 1$. Since $C_p$ is minimal for $\kerlin$, there is a subfunctor $F\subseteq \kerlin$ isomorphic to $S_{C_p, 1}$. Let $M:= \res^\mathcal{C}_\mathcal{D}\kerlin/F$. Then $M$ is nonzero only at $E_2$ and hence composition factors of $M$ are in bijection with the composition factors of $M(E_2)$ regarded as a $\mathbb C[\out(E_2)]$-module. We claim that
\[
\kerlin(E_2) = F(E_2) \oplus \mathcal K\kerlin(E_2),
\]
as $\mathrm{End}_\mathcal C(E_2)$-modules and hence $M(E_2) \cong \mathcal K\kerlin(E_2)$. In particular, a simple biset functor $S_{E_2, V}$ is a composition factor of $\kerlin$ if and only if $V$ is a composition factor of $\mathcal K\kerlin(E_2)$. To prove the claim, note that the sum  $F(E_2) + \mathcal K\kerlin(E_2)$ is a direct sum since $F(E_2)$ is a simple module which is not annihilated by the ideal $I_{E_2}$ whereas any submodule of $\mathcal K\kerlin(E_2)$ is annihilated by this ideal. Hence it is sufficient to compare the dimensions. 

By Corollary \ref{cor:ker-ipot}, $\kerlin(E_2)$ has basis consisting of the primitive idempotents $e_{H, h}^{E_2}$ such that $H\not = \langle h\rangle$. The number of such idempotents is $p^2 + p + 1$. On the other hand, by the main theorem of \cite{Bouc3}, the dimension of $S_{C_p, 1}(E_2)$ is $2p + 2$. Hence for the claim to be true, the dimension of $\mathcal K\kerlin(E_2)$ must be equal to $p^2-p-1$. 

Now we determine a basis for $\mathcal K\kerlin(E_2)$. Since it is contained in the intersection of kernels of restriction maps, we have that $\mathcal K\kerlin(E_2)\subseteq \mathrm{span}\{ e_{E_2, g}^{E_2}\mid g\in E_2\}$. Hence it is sufficient to determine the elements in this span which are annihilated by all deflation maps $\defl^{E_2}_{E_2/T}$ for $1< T< E_2$. By Lemma \ref{lem:defnoelem}, for any $g\in E_2$, we have
\[
p\cdot m_{E_2, g}^T =  \begin{cases}
              0 & \text{if } g\not\in T,\\
            1-p & \text{if } g = 1,\\
            1 & \text{if } 1\not = g\in T
          \end{cases} 
\]
Hence
\[
p\cdot\defl^{E_2}_{E_2/T}\Big(\sum_{g\in E_2}e_{E_2, g}^{E_2}\Big) =
(1-p)e_{E_2/T, T}^{E_2} + \sum_{g\in E_2-\{1\}} e_{E_2/T, T}^{E_2} = 0
\] 
that is, the sum $F_1:= \sum_{g\in E_2}e_{E_2, g}^{E_2}$ is in the intersection kernel. Note that if $p=2$, then $p^2 - p - 1$ is equal to $1$ and hence we conclude that $\mathcal{K}\kerlin(E_2)$ is of dimension 1. Hence the claim holds.

\begin{nothing} Suppose $p\ge 3$. We introduce other elements that lie in the restriction kernel. Note that by the above formula for the deflation numbers, it is clear that a sum of the form 
\[
\sum_{1\not = g\in T}\alpha_g e_{E_2, g}^{E_2}
\]
is in the restriction kernel if and only if the sum of coefficients $\alpha_g$ is zero. We may use irreducible characters of $\Aut(T)$ to produce such coefficients.

We fix our notation. Since $T=\langle t\rangle\cong C_p$, we have that $\Aut(T)\cong C_{p-1}$. Let $s$ be a primitive root modulo $p$ and $u$ be its multiplicative inverse modulo $p$. Define $\gamma_T: t\mapsto t^s$ which is a generator of $\Aut(T)$. Now for each non-trivial $\phi\in\widehat{\Aut(T)}$, define
\[
F_{t, \phi} = \sum_{\gamma\in\Aut(T)} \overline{\phi(\gamma)}\cdot e_{{E_2}, \gamma(t)}^{E_2}
\]
By orthonormality of the rows of the character table, the sum of the coefficients in the above sum is zero and hence $F_{t, \phi}$ is contained in the intersection kernel. Since the matrix of coefficients for $F_{t, \phi}$ as $\phi$ runs over all non-trivial irreducible characters of $\Aut(T)$ is a part of the character table of $\Aut(T)$, its rows are linearly independent and hence the set $\{ F_{t, \phi}\mid 1\not= \phi\in\widehat{\Aut(T)} \}$ is also linearly independent and has cardinality equal to $p-2$. 

We may repeat the same construction for each subgroup of order $p$ in $E_2$ and each such subgroup would produce a linearly independent set of $p-2$ elements. Notice that the union set is also linearly independent since the primitive idempotents that appear for different subgroups are distinct. Hence we obtain $(p+ 1)(p-2) = p^2 - p -2$ linearly independent elements in the restriction kernel. Together with $F_1$ above, which is clearly independent from the $F_{t, \phi}$, we obtain $p^2-p-1$ linearly independent elements in $\mathcal{K}\kerlin(E_2)$ as required. In particular, the set 
\begin{equation}\label{eqn:basisB}
\mathcal B = \{F_1\}\cup \bigcup_{1<\langle t\rangle < E_2} \{ F_{t, \phi}\mid 1\not =\phi\in\widehat{\Aut(\langle t\rangle)} \}
\end{equation}
is a basis for $\mathcal{K}\kerlin(E_2)$. With this result, we may restate Corollary \ref{cor:resker1} as follows.
\end{nothing}
\begin{corollary}\label{cor:resker2}
    Let $S$ be a simple biset functor with minimal group $H$ and put $V = S(H)$. Then the multiplicity of $S$ as a composition factor of $\kerlin$ as a biset functor is equal to the multiplicity of $V$ as a composition factor of $\mathcal K\kerlin(H)$ as a $\mathbb C[\out(H)]$-module.
\end{corollary}

From the proofs in this section we also obtain the complete list of all finite $p$-groups $G$ such that $\mathcal K\kerlin(G)\not =0$; completing the proof of Theorem \ref{thm:main1}.

\section{Calculation of Restriction Kernels I: Cyclic case}\label{sec:restker}

In this section we calculate the $\mathbb C[\out(G)]$-module structure of the restriction kernel $\mathcal K\kerlin(G)$ of $\kerlin$ for cyclic \( p \)-groups $G = C_{p^k}$ for all $k\ge 1$.

\begin{nothing} {\textbf{Notation.}}
Throughout this section, fix an integer $k\ge 1$ and put $G = C_{p^k}$.
When $k=1$, the evaluation $\kerlin(C_p)$ is the trivial $\mathbb C[\Aut(C_p)]$-module, so there is nothing to prove. This is also part (a) of Theorem \ref{thm:main}. We assume $k\ge 2$.
Denote by $H_l:=\langle h_l\rangle$ the subgroup of $G$ of order $p^l$. For any $0\le l<k$ and $\phi\in \widehat{\Aut(H_l)}$, define
\[
e_{G, l, \phi} := \sum_{\alpha\in\Aut(H_l)}\phi(\alpha^{-1})\cdot e_{G, \alpha(h_l)}^G
\]
and
\[
\mathcal D_G=\{ e_{G, l, \phi} \mid 0\le l < k,\ \phi\in\widehat{\Aut(H_l)}\}.
\]
\end{nothing}

\begin{nothing} {\textbf{Automorphism groups.}}
Let $H$ be a cyclic $p$-group of order $\ge p^2$ and $Z$ its unique minimal subgroup. The map
\[
\Psi: \Aut(H)\to \Aut(H/Z),\ \alpha\mapsto (\alpha_Z: xZ\mapsto \alpha(x)Z)
\]
is a surjective group homomorphism. We write $\Aut_Z(H):= \ker\Psi$, which has order $p$, and obtain
\[
\Aut(H)/\Aut_Z(H) \cong \Aut(H/Z) \text{ via } \alpha\Aut_Z(H)\mapsto \alpha_Z.
\]
\end{nothing}

\begin{theorem}\label{thm:cyclic}
Let $G = C_{p^k}$ with $k\ge 2$. The set $\mathcal{D}_G$ is a $\mathbb C$-basis for 
$\widetilde{e_G^G}\kerlin(G)$, on which $\Aut(G)$ acts monomially via
\begin{equation}\label{eqn:phistar}
\beta\cdot e_{G, l, \phi} = \phi(\beta|_{H_l})\, e_{G, l, \phi}
\end{equation}
for $\beta\in\Aut(G)$ and $e_{G,l,\phi}\in\mathcal{D}_G$. Moreover there is an isomorphism 
of $\mathbb C[\Aut(G)]$-modules
\begin{equation}\label{eqn:cyclicfinal}
\mathcal K\kerlin(G) \cong \bigoplus_{\substack{2\le l< k\\ \phi_p\,\mathrm{faithful}}}
\mathbb C[e_{G, l, \phi}]
\oplus \bigoplus_{\phi\not = 1}\mathbb C[e_{G, 1, \phi}]
\oplus\, \mathbb C[e_{G, 1, 1} - (1-p)e_{G, 0, 1}],
\end{equation}
and $\mathcal{K}\kerlin(G)$ is isomorphic to the direct sum of all non-primitive irreducible
$\mathbb C[\Aut(G)]$-modules, each with multiplicity $1$.
\end{theorem}

\noindent With this result, the following corollary is immediate.
\begin{corollary}
Part \textnormal{(b)} of Theorem~\ref{thm:main} holds.
\end{corollary}
\begin{proof}
By Corollary~\ref{cor:resker1}, the multiplicity of $S_{G,V}$ as a composition factor of $\kerlin$ equals the multiplicity of $V$ as a composition factor of $\mathcal{K}\kerlin(G)$ as a $\mathbb C[\out(G)]$-module. The result therefore follows directly from Theorem~\ref{thm:cyclic}.
\end{proof}
\noindent Now we prove the theorem.
\begin{proof}{[of Theorem ~\ref{thm:cyclic}.]}
The proof proceeds in four steps. We first show that $\mathcal{D}_G$ is a $\mathbb C$-basis for $\widetilde{e_G^G}\kerlin(G)$ and determine the $\Aut(G)$-action on it. We then compute the deflation formula for the basis elements $e_{G,l,\phi}$, from which the isomorphism (\ref{eqn:cyclicfinal}) is read off directly. Finally we verify that the summands afford exactly the non-primitive irreducible characters of $\Aut(G)$, each with multiplicity $1$.

\textbf{Step 1.} The dimension of $\widetilde{e_G^G} \kerlin(G)$ is equal to the number of non-generators of $G$ and it is also easy to see that the cardinality of $\mathcal{D}_G$ is also equal to this number. Thus it is sufficient to check that the set is linearly independent.

If $l\not = n$, then for any choices of $\phi\in\widehat{\Aut(H_l)}$ and $\psi\in\widehat{\Aut(H_n)}$, the primitive idempotents that appear in $e_{G, l, \phi}$ and $e_{G, n, \psi}$ are distinct. Hence there can be no linear relation between these type of elements. Hence the only possible relations are among those elements for fixed $l$. On the other hand, for a fixed $l$, the coefficient matrix for $e_{G, l, \phi}$ as $\phi$ runs over all elements of $\widehat{\Aut(H_l)}$ is equal to the character table of $\widehat{\Aut(H_l)}$, and hence is invertible. In particular the set $\{ e_{G, l, \phi} \mid \phi\in\widehat{\Aut(H_l)}\}$ is linearly independent, which proves the first claim.

For the second claim, the group $\Aut(G)$ acts on the primitive idempotents of $\MonG$ via
\[
\beta\cdot e_{H, h}^G = e_{\beta(H), \beta(h)}^G
\]
for $\beta\in\Aut(G)$ and $e_{H, h}^G\in \MonG$. This action gives rise to a monomial action on $\mathcal{D}_G$. We evaluate this action explicitly: Let $e_{G, l, \phi}\in \mathcal{D}_G$. Then
\[
\beta\cdot e_{G, l, \phi} = \sum_{\alpha\in\Aut(H_l)} \phi(\alpha^{-1})\cdot e_{G, \beta(\alpha(h_l))}^G
\]
Since $G$ is a cyclic group, the restriction of $\beta$ to $H_l$ is an automorphism of $H_l$ and hence $\beta\circ \alpha$ runs over all automorphisms of $\Aut(H_l)$ as $\alpha$ does. Therefore putting $\gamma = \beta\circ \alpha$ we get
\begin{eqnarray*}
\beta\cdot e_{G, l, \phi} &=& \sum_{\gamma\in\Aut(H_l)} \phi(\gamma^{-1}\circ \beta)\cdot e_{G, \gamma(h_l)}^G\\
&=& \phi(\beta|_{H_l}) \cdot e_{G, l, \phi}.
\end{eqnarray*}
Here the last equality holds since $\phi$ is a group homomorphism. Therefore the line $\mathbb C[e_{G, l, \phi}]$ in $\kerlin(G)$ is a $\mathbb C[\Aut(G)]$-submodule with character given by
\begin{equation}\label{eqn:phistar}
\Aut(G)\to \mathbb C^\times,\ \beta\mapsto \phi(\beta|_{H_l}).
\end{equation}

\smallskip
\noindent\textbf{Step 2.} By the transitivity of deflation maps, it is sufficient to determine the kernel of $\defl^G_{G/Z}$ where $Z:= H_1$ is the unique subgroup of $G$ of order $p$. Also note that any linear relation between deflated basis elements $\defl^G_{G/Z}e_{G, l, \phi}$ may be reduced to a relation for a fixed $l$. Indeed, for distinct numbers $l\not = n, l\ge 2$, we have, for $\alpha\in \Aut(H_l)$ and $\alpha'\in\Aut(H_n)$, that $\alpha(h_l)Z\not = \alpha'(h_n)Z$. Hence the sets of idempotents that appear in deflation of basis elements for $l$ and for $n$ are disjoint and hence no relation between them is possible. Thus we may work with a fixed index $l$.

We claim that
\begin{equation}\label{eqn:deftoZ}
  \defl^G_{G/Z} e_{G, l, \phi} =\begin{cases}
              0 & \text{if } l\ge 2 \text{ and }\ker\phi\cap \Aut_Z(H_l) = \{1\},\\
            e_{G/Z, l-1, \tilde\phi} & \text{if } l\ge 2 \text{ and }\ker\phi\cap \Aut_Z(H_l) \not = \{1\},\\
                          0 & \text{if } l = 1 \text{ and }\phi \not = 1,\\
            \frac{p-1}{p}e_{G/Z, 0, 1} & \text{if } l = 1 \text{ and } \phi = 1,\\
            \frac{1}{p}e_{G/Z, 0, 1} & \text{if } l = 0.
          \end{cases}
\end{equation}
First we assume $l\ge 2$. Note that since $G$ is cyclic, there is no supplement to $Z$ in $G$ and hence the deflation number $m_{G, g}^Z$ is equal to $1/p$ for any $g\in G$. Hence we have
\[
p\cdot \defl^G_{G/Z} e_{G, l, \phi} = \sum_{\alpha\in\Aut(H_l)}\phi(\alpha^{-1}) e_{G/Z, \alpha(h_l)Z}^{G/Z}.
\]
Using the above notation, $\alpha(h_l)Z$ is equal to $\Psi(\alpha)(h_lZ)$. Moreover for any $\kappa_z\in\Aut_Z(H_l)$, we have $\Psi(\alpha) = \Psi(\alpha\kappa_z)$. Hence the right hand side of the above equality can be rearranged as
\begin{eqnarray*}
&=& \sum_{\alpha\in [\Aut(H_l)/\Aut_Z(H_l)]} \Big(\sum_{\kappa\in\Aut_Z(H_l)}\phi(\alpha^{-1}\kappa^{-1})\Big) e_{G/Z, \Psi(\alpha)(h_lZ)}^{G/Z}
\end{eqnarray*}
where the first sum is over a complete set of left coset representatives of $\Aut_Z(H_l)$ in $\Aut(H_l)$. Using the multiplicativity of $\phi$, the right hand side of the above equation becomes
\begin{eqnarray*}
&=& \Big(\sum_{\kappa\in\Aut_Z(H_l)}\phi(\kappa^{-1})\Big)\sum_{\alpha\in [\Aut(H_l)/\Aut_Z(H_l)]} \phi(\alpha^{-1}) e_{G/Z, \Psi(\alpha)(h_lZ)}^{G/Z}.
\end{eqnarray*}
Notice that the coefficient is zero unless $\phi|_{\Aut_Z(H_l)}$ is the trivial character. Moreover when $\phi|_{\Aut_Z(H_l)}$ is the trivial character, it deflates to the quotient $\Aut(H_l)/\Aut_Z(H_l)$. We denote the image of this deflation under the isomorphism $\Psi$ by $\tilde\phi$. Then the above sum becomes
\begin{eqnarray*}
&=& p\sum_{\alpha\in [\Aut(H_l)/\Aut_Z(H_l)]} \tilde\phi(\alpha^{-1}_Z) e_{G/Z, \alpha_Z(h_lZ)}^{G/Z}\\
&=& p\sum_{\alpha_Z\in \Aut(H_l/Z)} \tilde\phi(\alpha^{-1}_Z) e_{G/Z, \alpha_Z(h_lZ)}^{G/Z}\\
&=& p\cdot e_{G/Z, l-1, \tilde\phi}.
\end{eqnarray*}
This proves the first two lines of Equation (\ref{eqn:deftoZ}). The remaining cases are not difficult to deal with. We leave the details to the reader. When $l=1$, we have $H_1 = Z$ and it is easy to see that the above result holds at this extreme and hence we obtain the third and the fourth lines of Equation (\ref{eqn:deftoZ}). The last line also follows since $e_{G/Z, 0, 1} = e_{G/Z, 1}^{G/Z}$.

\smallskip
\noindent\textbf{Step 3.}
From Equation (\ref{eqn:deftoZ}), the basis elements $e_{G, 1, \phi}$ for $\phi\not=1$ together with the combination $e_{G, 1, 1} - (1-p)\cdot e_{G, 0, 1}$ form a basis for the kernel of $\defl^G_{G/Z}$, and the isomorphism (\ref{eqn:cyclicfinal}) follows.

\smallskip
\noindent\textbf{Step 4.}
We need to show that each non-primitive irreducible character of $\Aut(G)$ appears on the right-hand side of Equation (\ref{eqn:cyclicfinal}) with multiplicity 1. First note that, by Equation \ref{eqn:phistar}, the character of the last term is the trivial character. For each $1\le l\le k$ and $\phi\in\widehat{\Aut(H_l)}$, let $\phi^*$ be the character of $\widehat{\Aut(G)}$ given by Equation (\ref{eqn:phistar}). It is clear that if we fix $l$, then $\phi^* = \psi^*$ if and only if $\phi = \psi$.

Fix a pair $(l, \phi)$. Also let $\alpha_p$ be a generator of the Sylow $p$-subgroup of $\Aut(G)$. Then since the $p$-part of $\phi$ is faithful, the value $\phi^*_p(\alpha_p) = \phi_p(\alpha_p\mid_{H_l})$ must be a primitive $p^{l-1}$-th root of unity. In particular, if $(m, \psi)$ is another pair with $l\not= m$, then we also have $\phi^*\not=\psi^*$. Thus there is no repetition in the right hand side of Equation (\ref{eqn:cyclicfinal}).

Moreover since the order of the restriction map $\Aut(G)\to \Aut(H_l)$, $\beta\mapsto \beta\mid_{H_l}$, is $p^{k-l}$, the induced character $\phi^*$ is non-primitive (since $p^{k-l}$ is a quasi-period for $\phi^*$).

We count the number of terms. If $2\le l< k$, then the number of irreducible characters of $\Aut(H_l)$ with faithful $p$-part is $p^{l-2}(p-1)^2$. Summing over $l$, we get
\begin{align*}
\dim_\mathbb C\Big(\bigoplus_{\substack{2\le l< k\\ \phi\in \widehat{\Aut(H_l)}\\ \phi_p:\text{faithful}}}\mathbb C[e_{G, l, \phi}]\Big)&=\sum_{\substack{2\le l< k\\ \phi\in \widehat{\Aut(H_l)}\\ \phi_p:\text{faithful}}}1 \\ &= \sum_{l=2}^{k-1} p^{l-2}(p-1)^2 = (p-1)(p^{k-2}-1).
\end{align*}
With the contribution of the last two terms, the dimension of the right hand side is
\[
(p-1)(p^{k-2}-1) + (p-2) + 1 = (p-1)p^{k-2},
\]
which is equal to the number of non-primitive irreducible characters of $\Aut(G)$. Hence each non-primitive irreducible $\mathbb C[\Aut(G)]$-module appears as a summand in Equation (\ref{eqn:cyclicfinal}) with multiplicity 1, as required.
\end{proof}

\section{Calculation of Restriction Kernels II: The $C_p\times C_p$-case}\label{sec:restker2}

In Corollary \ref{cor:resker2}, we have already determined the basis $\mathcal B$ given in \ref{eqn:basisB} for the restriction kernel $\mathcal{K}\kerlin(E_2)$ for $E_2 = C_p\times C_p$.

In this section we determine its $\mathbb C[\Aut(E_2)]$-module structure. It is clear that the basis vector $F_1$ generates a trivial $\mathbb C[\Aut(E_2)]$-submodule and a complement of the submodule is generated by the rest
\[
\mathcal B':= \mathcal B - \{F_1\} = \bigcup_{1<\langle t\rangle < E_2} \{ F_{t, \phi}\mid 1\not =\phi\in\widehat{\Aut(\langle t\rangle)}\}
\]
of the basis vectors. Hence we concentrate on the span of $\mathcal{B}'$. We begin by describing the structure of $\Aut(E_2)$.

\begin{nothing} {\textbf{Automorphism group of $E_2$.}}
We identify $\Aut(E_2) = \gl$ by writing $E_2 = \langle x\rangle \times \langle y\rangle$ as column vectors and representing each automorphism $\alpha$ by the matrix whose columns are $\alpha(x)$ and $\alpha(y)$; so $\alpha(x) = x^ay^c$ and $\alpha(y) = x^by^d$ for $\alpha = \begin{pmatrix} a & b \\ c & d \end{pmatrix}$.

The proper non-trivial subgroups of $E_2$ each have order $p$ and are cyclic. We index them via $\mathfrak C = \{x\} \cup \{x^iy \mid 0 \le i \le p-1\}$ by setting $T_g = \langle g\rangle$ for $g \in \mathfrak C$, so that
\[
\mathfrak S = \{ T_g \mid g \in \mathfrak C \} = \{ T_x, T_y, T_{xy}, T_{x^2y}, \ldots, T_{x^{p-1}y} \}.
\]
For $g\in\mathfrak C$, we write $\Gamma_g = \Aut(T_g)$ and $\hat\Gamma_g$ for its character group. We fix a primitive root $s$ modulo $p$ and let $\gamma_g \in \Gamma_g$ denote the $s$-th power map; this is canonical since it is the restriction of the scalar matrix $\begin{pmatrix} s & 0 \\ 0 & s \end{pmatrix}$ to $T_g$, which commutes with every element of $\gl$. Consequently, for any $\alpha\in\gl$,
\[
\alpha \circ \gamma_g = \gamma_{\alpha(g)} \circ \alpha,
\]
so the isomorphism $\tilde\alpha\colon \Gamma_g \to \Gamma_{\alpha(g)}$ induced by conjugation by $\alpha$ sends $\gamma_g$ to $\gamma_{\alpha(g)}$ and depends only on the subgroup $T_{\alpha(g)}$, not on the particular $\alpha$. Similarly the induced map $\hat\alpha\colon\hat\Gamma_g\to\hat\Gamma_{\alpha(g)}$, $\phi\mapsto{}^\alpha\phi$ (defined by ${}^\alpha\phi({}^\alpha\gamma)=\phi(\gamma)$), depends only on $T_{\alpha(g)}$.

The $\gl$-action on $\mathfrak S$ is given by $\alpha\cdot T_g = T_{\alpha(g)}$. Explicitly, for $\alpha = \begin{pmatrix} a & b \\ c & d \end{pmatrix}$,
\[
\alpha\cdot T_{x^iy} = T_{x^{(ai+b)f}y}, \qquad
\alpha\cdot T_x = \begin{cases} T_{x^{ac^{-1}}y} & \text{if } c\neq 0,\\ T_x & \text{otherwise,}\end{cases}
\]
where $f = (ci+d)^{-1}$ in $\mathbb F_p$. This action is transitive, and the stabilizer of $T_x$ is the Borel subgroup $B$ of upper triangular matrices.
\end{nothing}

\begin{nothing}
Now we are ready to determine $\mathcal{K}\kerlin(E_2)$ as a $\mathbb C[\gl]$-module. We write
\[
\mathbb C[\mathcal B]:= \{ \mathbb C[F_{g, \phi}]\mid F_{g, \phi}\in \mathcal B\}
\]
for the set of lines in $\kerlin(G)$ parameterized by the set $\mathcal{B}$.
\end{nothing}
\begin{theorem}\label{thm:cpcpfinal}
    Assume the above notation.

    {\rm (a)} The set $\mathbb C[\mathcal{B}]$ is a monomial $\gl$-set via the action induced from the action of $\gl$ on $\kerlin(G)$.

    \smallskip
    {\rm (b)} There is a unique fixed point in $\mathcal B$, namely $F_1$.
    
    \smallskip
    {\rm (c)} The $\gl$-orbits in $\mathbb C[\mathcal B]$ of size $\ge 2$ are in one-to-one correspondence with the non-trivial irreducible characters of $\Gamma_x$.
    
    \smallskip
    {\rm (d)} The stabilizer in $\gl$ of any of these orbits is conjugate to $B\le\gl$.
    
    \smallskip
    {\rm (e)} The $B$-action on the line $\mathbb C[F_{x, \phi}]$ is given by the linear character $\tilde\phi: B\to \mathbb C^\times, \begin{pmatrix}
        a & b\\ 0 & d
    \end{pmatrix}\mapsto \phi(a)$ of $B$. 
\end{theorem}
\begin{proof} We first note that part Part {\rm (b)} is trivial. For the rest of the theorem, we begin by proving that the action of $\gl$ on $\kerlin(E_2)$ induces a permutation action on $\mathbb C[\mathcal B]$. For simplicity, we consider the $\gl$-action on $F_{x, \phi}$ for a fixed character $\phi\in\hat\Gamma_x$. Let $\alpha\in\gl$. Then
\begin{eqnarray*}
    \alpha\cdot F_{x,\phi} &=& \sum_{\sigma\in\Aut(T_x)}\overline{\phi(\sigma)} \alpha(e_{G, \sigma(x)}^G)\\
    &=& \sum_{\sigma\in\Aut(T_x)}\overline{\phi(\sigma)} e_{G, \alpha\sigma(x)}^G\\
    &=& \sum_{\delta\in\Aut(T_{\alpha(x)})}\overline{\phi(\alpha\delta\alpha^{-1})} e_{G, \delta\alpha(x)}^G\\
    &=& \sum_{\delta\in\Aut(T_{\alpha(x)})}\overline{{}^\alpha\phi(\delta)} e_{G, \delta\alpha(x)}^G
\end{eqnarray*}
We divide the rest of the calculation into two cases. First suppose $\alpha = \begin{pmatrix}
    a & b\\ 0 & d
\end{pmatrix}\in B$. Then $\alpha(x) = x^a$ and hence $\alpha' :=\alpha\mid_{T_x}\in\Gamma_x$. In particular $T_{\alpha(x)} = T_x$ and $\delta':= \delta\alpha'\in\Gamma_x$ for any $\delta\in\Gamma_x$. The above sum becomes
\begin{eqnarray*}
  \alpha\cdot F_{x,\phi} &=&\sum_{\delta'\in\Aut(T_{x})}\overline{\phi(\delta'(\alpha')^{-1})} e_{G, \delta'(x)}^G\\
&=&\phi(\alpha')\sum_{\delta'\in\Aut(T_{x})}\overline{\phi(\delta'} e_{G, \delta'(x)}^G\\
&=& \phi(\alpha') F_{x, \phi}.
\end{eqnarray*}
This calculation also shows that the subgroup $B$ stabilizes the line $\mathbb C[F_{x, \phi}]$. Since it is the stabilizer of $T_x$, we conclude that 
\[
B = \text{Stab}_{\gl}(\mathbb C[F_{x, \phi}])
\]
Moreover we obtain that if $\phi\not = \psi$, then the lines $\mathbb C[F_{x, \phi}]$ and $\mathbb C[F_{x, \psi}]$ are not in the same $\gl$-orbit.

Next suppose $\alpha = \begin{pmatrix}
    a&b\\c&d
\end{pmatrix}$ with $c\not = 0$. Then $T_{\alpha(x)} = T_{x^{ac^{-1}}y}$. Also let $\delta_c: T_{\alpha(x)}\to T_{\alpha(x)}, g\mapsto g^c$. Then 
\begin{eqnarray*}
  \alpha\cdot F_{x,\phi} &=& 
\sum_{\delta\in\Aut(T_{\alpha(x)})}\overline{{}^\alpha\phi(\delta)} e_{G, \delta\alpha(x)}^G\\
&=& \sum_{\delta'\in\Aut(T_{\alpha(x)})}\overline{{}^\alpha\phi(\delta'\delta_c^{-1})} e_{G, \delta'(x^{ac^{-1}}y)}^G\\
&=& {}^\alpha\phi(\delta_c)\sum_{\delta'\in\Aut(T_{\alpha(x)})}\overline{{}^\alpha\phi(\delta')} e_{G, \delta'(x^{ac^{-1}}y)}^G\\
&=& {}^\alpha\phi(\delta_c)F_{x^{ac^{-1}y}, {}^\alpha\phi}
\end{eqnarray*}
It is now clear that the group $\gl$ acts on $\mathbb C[\mathcal{B}]$ by permutations. By the previous case, we already know that there are at least $p-2$-orbits in $\mathbb C[\mathcal{B}]$ and each orbit is isomorphic to the coset space $\gl/B$. Hence these orbits cover at least $(p+1)(p-2)$ elements. But the total number of elements is $1+(p+1)(p-2)$. The additional orbit comes from the fixed line $\mathbb C[F_1]$. 

\end{proof}

\begin{nothing}
    By the above theorem for each non-trivial irreducible character $\phi$ of $\Gamma_x$, there is a direct summand of $\mathcal{K}\kerlin(E_2)$ isomorphic to the induced $\mathbb C[\gl]$-module $\ind_B^{\gl}\phi_1$ where $\phi_1: B\to \mathbb C^\times$ is given by 
    \[
    \phi_1\begin{pmatrix}
        a & b\\ 0& d
    \end{pmatrix} = \phi(a).
    \]
    By the classification of irreducible characters of $\gl$, given in Theorem 8.12 in \cite{Shapiro}, this module is irreducible. Hence we obtain the following result from which Theorem \ref{thm:main}(c) follows.
\end{nothing}
\begin{theorem}[Structure of \( \mathcal{K}\kerlin(G) \) for \( G \cong C_p \times C_p \)]\label{thm:cpcpfinal2}
    Assume the above notation. 
    The restriction kernel \( \mathcal{K}\kerlin(G) \) decomposes as:
\[
\mathcal{K}\kerlin(G) \cong \begin{cases}
\mathbb{C} \oplus \bigoplus_{1\neq \phi\in \hat\Gamma_x} \ind_B^{\gl}\phi_1, & \text{if } p > 2, \\
\mathbb{C}, & \text{if } p = 2,
\end{cases}
\]
Moreover the right-hand side of the equality is a direct sum of irreducible $\mathbb C[\gl]$-modules.
\end{theorem}

\section{Calculation of Restriction Kernels III}\label{sec:restkerGen}

This section analyzes the restriction kernel \( \mathcal{K}\kerlin(G) \) for \( G \cong C_{p^m} \times C_p \), where \( m \geq 2 \). Its structure as a module over the automorphism group $\Gamma:= \Aut(G)$ is more complicated and we need to study the cases $p=2$ and $p>2$ separately. In this section we include results that apply to both cases and in the next two sections we calculate explicit summands. 

Since $\mathcal K\kerlint(G)\subseteq \delta_\Phi\kerlint(G)$, we may concentrate on $\delta_\Phi\kerlint(G)$. Assume the notation from Section \ref{sec:reduc1}.
By the proof of Theorem \ref{pro:1Cp1Cpk}, the later is generated by the set $\mathcal E:=\{ E_g:= \varphi_1^G\cdot e_{G, g}^G\mid g\in G\}$. We partition the set $\mathcal E$ according to the three cases of the proof of Proposition \ref{pro:1Cp1Cpk} as follows: Let

\smallskip
{\rm (a)} $\mathcal E_1 = \big\{ E_g \mid \langle g\rangle \text{ is maximal in } G \big\}$

\smallskip
{\rm (b)} $\mathcal E_2 = \big\{ E_g \mid g\in\Phi(G) \big\}$

\smallskip
{\rm (c)} $\mathcal E_3 = \big\{ E_g \mid g\not\in\Phi(G), \langle g\rangle \text{ is not maximal in } G \big\}$

\noindent then
\[
\mathcal E =  \mathcal E_1 \cup\mathcal E_2  \cup \mathcal E_3.
\]
With this notation and by Case (1) of the proof of Proposition \ref{pro:1Cp1Cpk}, span of $\mathcal E_1$ lies in $\mathcal K\kerlint(G)$. Moreover the following lemma holds.
\begin{theorem}
    There is an isomorphism 
    \[
    \mathcal K\kerlin(G) \cong \text{Span }(\mathcal E_1)
    \] of vector spaces over $\mathbb C$.
\end{theorem}
\begin{proof} We prove the proposition in three steps. In Step 1 and Step 2, we prove that deflation to a non-trivial quotient of $G$ is injective on each $Z$-orbit in $\mathcal E_2$ and also on each $Z$-orbit in $\mathcal E_3$. Then Step 3 proves that the images are disjoint, and hence deflation is injective on the span of the union. Combining these, we get that 
\[
\mathcal K\kerlin(G)\cap \big(\text{span}_\mathbb C(\mathcal{E}_2)\oplus \text{span}_\mathbb C(\mathcal{E}_3)\big) = \{0\} 
\]
and hence 
\[
\mathcal K\kerlin(G) = \text{span}_\mathbb C(\mathcal{E}_1). 
\]

\noindent {\bf Step 1:}
The set $\mathcal E_2$ can be decomposed further. Notice that for any $g\in G$, the primitive idempotent $e_{G, g}^G$ is a summand of each of $E_{gz^i}$ and not a summand of any other $E_h$'s. Hence it is sufficient to look for linear relations including $E_g$ only among $E_{gz^i}$'s. Accordingly we partition $\mathcal E_2$ as
\[
\mathcal E_2 = \bigsqcup_{g\in [\Phi(G)/Z]}\Big\{ E_{gz^i}\mid i=0, 1, \ldots, p-1 \Big\}
\]
where the union is over a complete set of representatives of cosets of $Z$ in $\Phi(G)$. We claim that for a fixed $g\in [\Phi(G)/Z]$, the set 
$\Big\{ E_{gz^i}\mid i=0, 1, \ldots, p-1 \Big\}$ generates a vector space of dimension $p-1$. 

If $p=2$, the claim is trivial since in this case $z$ has order 2 and $E_g = -E_{gz}$. So suppose $p$ is odd. Recall that, for a fixed $i$, we have
\[
E_{gz^i} = e_{G, gz^i}^G - \frac{1}{p}\sum_{j=0}^{p-1}e_{G, gz^iz^j}^G = \frac{p-1}{p}e_{G, gz^i}^G - \frac{1}{p}\sum_{j=1}^{p-1}e_{G, gz^iz^j}^G
\]
Consider the matrix $M=(m_{i, j})_{0\le i,j<p}$ of coefficients for a linear system in $p\cdot E_{gz^i}$'s. We have
\[
m_{i, j} = \begin{cases}
              p-1 & \text{if } i = j,\\
              -1 & \text{ otherwise.}
          \end{cases}  
\]
The matrix $M$ is not invertible since the sum of all rows is equal to 0. We claim that the rank of $M$ is $p-1$. To see this, we add the negative of the first row of $M$ to each of the other rows. The resulting matrix $M' = (m_{i, j}')$ has entries
\[
m_{i, j}' = \begin{cases}
              m_{i, j} & \text{if } i = 1,\\
              p & \text{ if } i=j>1,\\
              -p & \text{ if } i>1, j=1\\
              0 & \text{otherwise.}
          \end{cases}  
\]
It is clear that the minor $M_{1,1}'$ is equal to $p^{p-1}$, hence $M'$ has rank $p-1$. On the other hand, since it is obtained from $M$ by multiplying $M$ by an invertible matrix, the rank is smaller than or equal to the rank of $M$. Since the rank of $M$ is strictly smaller than $p$, we must have $\text{rk}(M) = p-1$. 

Now let $T$ be a subgroup of order $p$ in $G$ which is not contained in $\Phi(G)$. Then, we have
\[
\frac{p^2}{p-1}\defl^G_{G/T} E_{gz^i} = (p-1)e_{G/T, gz^iT}^{G/T} - \sum_{j=1}^{p-1}e_{G/T, gz^iz^jT}^{G/T}
\]
Since $T$ is not contained in the Frattini subgroup, the elements $gz^iT$ of $G/T$ are all distinct and hence the image of deflation $\defl^G_{G/T}$ on the span of $\{E_{gz^i} \mid i = 0,\ldots,p-1\}$ is generated by the set
\[
\mathcal E_{2,g}' := \Big\{ \frac{p^2}{p-1}\defl^G_{G/T} E_{gz^i}\mid i=0, 1, \ldots, p-1 \Big\}
\]
consisting of distinct elements and has dimension $p-1$. In particular, $\defl^G_{G/T}$ is injective on $\mathrm{Span}\{E_{gz^i} \mid i = 0,\ldots,p-1\}$, as required.

\noindent {\bf Step 2:} Let $g \in \mathcal{E}_3$, so $g = x^a y^b$ with $b \neq 0$ and 
$p \mid a$. Write $a = pa'$. Let $gZ = \{gz^i \mid i = 0, 1, \ldots, p-1\}$ be the $Z$-orbit 
of $g$ in $\mathcal{E}_3$. We first verify that 
\[
gz^i\Phi(G) = y^b\Phi(G) \quad \text{in } G/\Phi(G) \cong C_p \times C_p
\]
for all $i = 0, 1, \ldots, p-1$. Indeed, $gz^i = x^{pa' + ip^{m-1}}y^b$, and since $p \mid pa'$ 
and $p \mid ip^{m-1}$ (as $m \geq 2$), the $x$-component lies in $\Phi(G) = \langle x^p\rangle$, 
giving $gz^i\Phi(G) = y^b\Phi(G)$ for all $i$.

Choose a subgroup $T$ of order $p$ such that $T \not\leq \Phi(G)$ and 
$T\Phi(G)/\Phi(G) = \langle y^b\Phi(G)\rangle$ in $G/\Phi(G)$. (Such $T$ exists; for 
instance, take $T = \langle x^{p^{m-1}}y^b \rangle$, which has order $p$ and satisfies 
$T \cap \Phi(G) = 1$ since the $y$-component $y^{kb}$ of any $(x^{p^{m-1}}y^b)^k$ is 
trivial only when $p \mid k$.) Since $T \not\leq \Phi(G)$, we have $T \neq Z$, so $T \cap Z = 1$ 
and in particular $z \notin T$.

We now compute $m_{G, gz^i}^T$ for each $i$. Since $|T \cap \Phi(G)| = 1$, Lemma 
\ref{lem:defnoFrat} and Lemma \ref{lem:defnoelem} give
\[
m_{G, gz^i}^T = m_{G/\Phi(G),\, gz^i\Phi(G)}^{T\Phi(G)/\Phi(G)} = m_{G/\Phi(G),\, y^b\Phi(G)}^{\langle y^b\Phi(G)\rangle} = \frac{1}{p}
\]
for all $i = 0, 1, \ldots, p-1$.

Next, since $z \notin T$, the cosets $gz^i T$ are pairwise distinct in $G/T$: indeed 
$gz^iT = gz^jT$ would force $z^{i-j} \in T$, hence $z^{i-j} \in T \cap Z = 1$, 
giving $i \equiv j \pmod{p}$.
Using Equation (\ref{eqn:phieee}) and the above calculation, we get
\[
p^2\, \defl^G_{G/T} E_{gz^i} = (p-1)\,e_{G/T,\, gz^iT}^{G/T} - \sum_{\substack{j=0 \\ j\neq i}}^{p-1} e_{G/T,\, gz^jT}^{G/T}.
\]
The matrix of this system in the (distinct) basis elements $e_{G/T,\, gz^jT}^{G/T}$ is 
exactly the same matrix $M$ as in Step 1, which has rank $p-1$ with 
one-dimensional null space spanned by the all-ones vector. 

Now observe that
\[
\sum_{i=0}^{p-1} E_{gz^i} 
= \sum_{i=0}^{p-1}\bigg(e_{G,\, gz^i}^G - \frac{1}{p}\sum_{k=0}^{p-1}e_{G,\, gz^{i+k}}^G\bigg)
= \sum_{i=0}^{p-1}e_{G,\, gz^i}^G - \frac{1}{p}\cdot p\sum_{i=0}^{p-1}e_{G,\, gz^i}^G = 0.
\]
Hence the all-ones vector maps to zero in $\mathrm{Span}\{E_{gz^i} \mid i = 0, \ldots, p-1\}$, 
which is therefore $(p-1)$-dimensional. Since the null space of $M$ is generated by the 
all-ones vector, the map $p^2\,\defl^G_{G/T}$ is injective on 
$\mathrm{Span}\{E_{gz^i} \mid i = 0, \ldots, p-1\}$, as required.

\noindent {\bf Step 3:} We first observe that the primitive
idempotents supporting $\mathcal{E}_2$ and $\mathcal{E}_3$ are
disjoint. Indeed, by Equation~(\ref{eqn:phieee}), each $E_g$ is supported on
$\{e_{G,gz^i}^G \mid i = 0, \ldots, p-1\}$. For $g \in \mathcal{E}_2$ (so
$g \in \Phi(G) = \langle x^p \rangle$), every $gz^i \in \Phi(G)$ has trivial
$y$-component. For $g' \in \mathcal{E}_3$ (so $g' = x^{pa'}y^b$ with $b \neq 0$),
every $g'z^i = x^{pa'+ip^{m-1}}y^b$ has nontrivial $y$-component. Since
$\Phi(G)$ consists precisely of elements with trivial $y$-component,
the two collections of primitive idempotents are disjoint, so
$\mathrm{Span}(\mathcal{E}_2) \oplus \mathrm{Span}(\mathcal{E}_3)$ is a direct sum.

Now let $u = u_2 + u_3 \in \mathrm{Span}(\mathcal{E}_2) \oplus \mathrm{Span}(\mathcal{E}_3)$
be an element of $\mathcal{K}\kerlin(G)$; we show $u = 0$.

\textbf{Step 3.1: $u_2 = 0$.} Choose a subgroup $T$ of order $p$ with
$T \not\leq \Phi(G)$ and $T\Phi(G)/\Phi(G) = \langle x\Phi(G) \rangle$ in 
$G/\Phi(G) \cong C_p \times C_p$.
For any $h$ appearing in the support of $u_3$, we have $h\Phi(G) = y^b\Phi(G) \in 
\langle y\Phi(G) \rangle$, which is disjoint from $\langle x\Phi(G) \rangle$ in 
$G/\Phi(G)$. By Lemma~\ref{lem:defnoFrat} and Lemma~\ref{lem:defnoelem} 
(rank $2$, element outside $T$):
\[
m_{G,h}^T = m_{G/\Phi(G),\, h\Phi(G)}^{\langle x\Phi(G)\rangle} = \frac{1 - p^0}{p} = 0.
\]
Hence $\defl^G_{G/T}(u_3) = 0$. Since $u \in \mathcal{K}\kerlin(G)$, we have
$\defl^G_{G/T}(u) = 0$, so $\defl^G_{G/T}(u_2) = 0$.
Moreover, since $g \in \mathcal{E}_2$ implies $gz^i \in \Phi(G)$ and 
$T \cap \Phi(G) = 1$, the cosets $gz^iT$ are distinct across all $Z$-orbits 
in $\mathcal{E}_2$ simultaneously. Hence $\defl^G_{G/T}$ is injective on the 
full $\mathrm{Span}(\mathcal{E}_2)$ by Step~1, and we conclude $u_2 = 0$.

\textbf{Step 3.2: $u_3 = 0$.} We now have $u = u_3 \in \mathcal{K}\kerlin(G)$,
so $\defl^G_{G/T}(u_3) = 0$ for every subgroup $T$ of order $p$ with 
$T \cap \Phi(G) = 1$. Choosing $T$ with $T\Phi(G)/\Phi(G) = \langle y\Phi(G) \rangle$, 
the same argument shows that the cosets $hz^iT$ are distinct across all $Z$-orbits 
in $\mathcal{E}_3$ simultaneously (since $T \cap \Phi(G) = 1$ and all $hz^i$ have 
the same nontrivial $y$-component within each orbit). Hence $\defl^G_{G/T}$ is 
injective on the full $\mathrm{Span}(\mathcal{E}_3)$ by Step~2, giving $u_3 = 0$.

Therefore $\mathcal{K}\kerlin(G) \cap \bigl(\mathrm{Span}(\mathcal{E}_2) \oplus
\mathrm{Span}(\mathcal{E}_3)\bigr) = \{0\}$, and since
$\mathrm{Span}(\mathcal{E}_1) \subseteq \mathcal{K}\kerlin(G) \subseteq
\delta_\Phi\kerlin(G) = \mathrm{Span}(\mathcal{E}_1) \oplus \mathrm{Span}(\mathcal{E}_2) 
\oplus \mathrm{Span}(\mathcal{E}_3)$, we conclude
$\mathcal{K}\kerlin(G) = \mathrm{Span}(\mathcal{E}_1)$.
\end{proof}

\begin{nothing} To determine a basis for $\mathcal K\kerlin(G)$, one needs to go over the above argument once more. We introduce a basis without repeating the details. Let $\mathcal M$ be the set of elements of $G$ of order $p^m$. The group $Z$ acts on $\mathcal M$ by left multiplication. Each $Z$-orbit $gZ$ in $\mathcal M$ gives $p$ elements in $\mathcal E_1$ and they generate a $(p-1)$-dimensional subspace, denoted by $\mathcal K_{gZ}$. A basis for $\mathcal K_{gZ}$ is given by $\{ E_{gz^i}\mid i=1, 2, \ldots, p-1\}$. Clearly $\mathcal K_{gZ}\cap \mathcal K_{hZ} = \{0\}$
if $gZ\not = hZ$. Hence 
\[
\mathcal K\kerlin(G) = \bigoplus_{gZ\in [\mathcal M/Z]} \mathcal K_{gZ}
\]
and the set 
\[
\mathcal B_p:=\big\{ E_{gz^i} \mid gZ\in [\mathcal M/Z], i=1, 2, \ldots, p-1\big\}
\]
is a basis for $\mathcal K\kerlin(G)$. Since the number of cyclic subgroups of order $p^m$ of $G$ is $p$ and each of these subgroups has $p^{m-1}(p-1)$ generators, we have 
\[
\dim_\mathbb C\mathcal K\kerlin(G) = p\cdot \frac{p^{m-1}(p-1)}{p} \cdot (p-1)= p^{m-1}(p-1)^2.
\]
\end{nothing}

\section{Calculation of Restriction Kernels IV: $C_{2^m}\times C_2$-case}\label{sec:restker3}

This section is devoted to the determination of the irreducible constituents of $\mathcal K\kerlint(C_{2^m}\times C_2)$ as a $\mathbb C\Gamma$-module where we set $\Gamma$ to be the group $\Aut(C_{2^m}\times C_2)$. We collect necessary results about the group $\Gamma$ in \ref{app:evenp}. For the rest of this section, we assume the notation from the appendix, whenever needed.

\begin{nothing} {\bf The case $G= C_4\times C_2$.} To begin with, we separate this case 
because the automorphism group of $G$ is an exception.  We have $\Gamma  \cong D_8$, the dihedral group of order 8 and this is not a special case of the presentation given in \cite{BC} (and also in the appendix). Put $C_4\times C_2 = \langle x,  y\mid x^4 = y^2 = 1, xy = yx\rangle$ and $\Gamma = \langle a, b\mid a^4 = b^2 = 1, bab = a^3\rangle$. Here we identify $a$ by the automorphism $x\mapsto xy, y\mapsto x^2y$ and $b$ as the automorphism $x\mapsto x^3, y\mapsto y$. The following theorem settles this case and proves a special case of Part (d) of Theorem \ref{thm:main}.
\end{nothing}

\begin{theorem}
    Assume the above notation. Then
    \begin{enumerate}
        \item[(i)] The set $\mathcal{B}_2$ is a transitive monomial basis for the $\mathbb C[\Gamma]$-module $\mathcal{K}\kerlint(G)$.
        \item[(ii)] We have \[\Sigma:= \text{Stab}_\Gamma(\mathbb C[E_x]) = \langle a^2, b\rangle\cong C_2\times C_2\] and for any $\sigma\in\Sigma$, the equality
        \[\sigma\cdot E_x = \phi(\sigma)\cdot E_x\]
        holds. Here $\phi = \infl_{\Sigma/\langle a^2b\rangle}^\Sigma (-1)$ where $(-1)$ is the non-trivial irreducible character of $\Sigma/\langle a^2b\rangle$.
        \item[(iii)] The $\mathbb C[\Gamma]$-module $\mathcal{K}\kerlint(G)$ is isomorphic to the unique two dimensional irreducible $\mathbb C[\Gamma]$-module. More precisely
        \[
        \mathcal{K}\kerlint(C_4\times C_2) \cong \ind_\Sigma^\Gamma\infl_{\Sigma/\langle a^2b\rangle}^\Sigma (-1).
        \]
    \end{enumerate}
\end{theorem}
\begin{proof}
With the above notation, the set $\mathcal M/Z$ of $Z$-orbits of $\mathcal{M}$ becomes
\[
\mathcal M = \big\{ \{x, x^3\}, \{xy, x^3y\}\big\}
\]
Hence the basis $\mathcal B_2$ consists of the two elements $E_x$ and $E_{xy}$ where 
\[
E_x = e_{G, x}^G - e_{G, x^3}^G \text{ and } E_{xy} = e_{G, xy}^G - e_{G, x^3y}^G
\]
The action of $\Gamma$ on $\mathcal{B}_2$ can be evaluated using the following table.
\begin{center}
    \begin{tabular}{c|cc}
         & $a$ & $b$  \\ \hline
         $E_x$ & $E_{xy}$ & $-E_{x}$  \\
         $E_{xy}$ & $-E_{x}$ & $-E_{xy}$  \\
           \end{tabular}
\end{center}
It is now clear that $\mathcal B_2$ is a monomial basis for $\mathcal{K}\kerlint(G)$ and the stabilizer $\Sigma$ of 
the line $\mathbb C[E_x]$ is the subgroup generated by $a^2$ and $b$. Moreover the action of $\Sigma$ on 
$\mathbb C[E_x]$ is given by the following homomorphism $\Sigma\to \mathbb C^\times$.
\begin{center}
    \begin{tabular}{c|cccc}
         & 1 & $a^2$ & $b$ & $a^2b$  \\ \midrule 
         $\phi:= \infl_{\Sigma/\langle a^2b\rangle}^\Sigma (-1)$ & $1$ & $-1$ & $-1$ & $1$  \\
           \end{tabular}
\end{center}
Here by $(-1)$ (and below by $\mathbb C_{(-1)}$) we denote the non-trivial irreducible character (module) of the group $\Sigma/\langle a^2b\rangle$. Therefore we obtain 
\[
\mathcal{K}\kerlint(G)\cong \ind_\Sigma^\Gamma \infl_{\Sigma/\langle a^2b\rangle}^\Sigma \mathbb C_{(-1)}= \ind_\Sigma^\Gamma C_\phi
\]
as $\mathbb C\Gamma$-modules. In particular $\mathcal{K}\kerlint(G)$ is isomorphic to the unique irreducible $\mathbb C\Gamma$-module of dimension 2.
\end{proof}

\begin{nothing} {\bf The case $m>2$.}
Now we assume $m>2$. In this case, the set $\mathcal B_2$ simplifies again since there are only 2 subgroups of order $2^m$ and the subgroup $Z=\langle z\rangle$ has order 2. Writing $G = \langle x, y\mid x^{2^m} = y^2 = 1, xy = yx\rangle$, we have $z= x^{2^{m-1}}$. With this notation, the list of elements of order $2^m$ is
\[
\mathcal M= \{ x^iy^j\mid 0<i< m, i\equiv 1 (\text{mod }2), j= 0,1 \}
\]
Also for any element $g\in\mathcal M$, we clearly have $E_g = -E_{gz}$. Therefore the line $\mathcal K_{gZ}$ can be represented as $\mathbb C[E_g]$ and a basis for $\mathcal K\kerlint(G)$ is given by 
\[
\mathcal B_2= \{ E_{x^iy^j}\mid 0<i<m/2 , i\equiv 1 (\text{mod }2), j= 0,1 \}
\]
We introduce some notation to use in the next lemma. Let $\Sigma = \langle a_2^{2^{m-3}}, b \rangle\le \Gamma$. This is a non-cyclic subgroup of order 4. We denote the non-trivial character of $\Sigma/\langle b\rangle$ by $-1$ and its inflation to $\Sigma$ by $\psi$.  
\end{nothing}
\begin{lemma}
The set \(\mathbb{C}[\mathcal{B}_2]\) forms a transitive monomial \(\Gamma\)-set, with the action of \(\Gamma\) defined via the isomorphism biset \(\operatorname{Iso}_G^a\) for each automorphism \(a \in \Gamma\). Moreover:
\begin{enumerate}
    \item[\textbf{(a)}] The stabilizer in \(\Gamma\) of \(\mathbb{C}[E_x]\) is \(\Sigma\).
    \item[\textbf{(b)}] The subgroup \(\Sigma\) acts on vector space \(\mathbb{C}[E_x]\) via the character \(\psi\).
\end{enumerate}
\end{lemma}

\begin{proof}
First, we verify that the \(\Gamma\)-action is monomial. Observe that the element \(z\) is fixed under any automorphism \(a \in \Gamma\). For a given \(a \in \Gamma\), we compute:
\[
\iso_G^a(E_g) = \iso_G^a(e_{G, g}^G - e_{G, gz}^G) 
= e_{G, a(g)}^G - e_{G, a(g)z}^G = \pm E_h,
\]
where \(h = a(g)\) or \(h = a(g)z\), depending on the action of \(a\). This shows that \(\iso_G^a(E_g)\) lies in the same orbit under \(\Gamma\), confirming that the action is monomial. Transitivity follows from the fact that \(\Gamma\) acts transitively on the set of generators.

Next, we determine the stabilizer subgroup \(\Sigma = \operatorname{Stab}_\Gamma(\mathbb{C} [E_x])\). From Appendix~\ref{app:evenp}, \(\Gamma\) is generated by the elements \(a_1\), \(a_2\), \(b\), and \(c\). Using the table in the appendix, we compute the action of these elements on \(E_x\):
\[
\begin{array}{c|cccc}
 & a_1 & a_2 & b & c \\ \hline
E_x & E_{x^{-1}} & E_{x^5} & E_x & E_{xy}.
\end{array}
\]

Note that, since \(m > 2\), the inverse \(x^{-1}\) is distinct from both \(x\) and \(xz\). Thus, \(E_{x^{-1}} \neq \pm E_x\). Similarly, it is evident that \(E_{xy} \neq \pm E_x\). Therefore, \(\langle a_1, c \rangle \not\subseteq \Sigma\).

On the other hand, since \(a_2\) has order \(2^{m-2}\), the element \(\tilde{a}_2 := a_2^{2^{m-3}}\) is in \(\Sigma\). Furthermore, for \(i\) not divisible by \(2^{m-3}\), \(a_2^i \notin \Sigma\). The action of \(\tilde{a}_2\) on \(E_x\) is given by:
\[
\tilde{a}_2(E_x) = -E_x,
\]
because \(5^{2^{m-3}} \equiv 1 + 2^{m-1}\).

Additionally, \(b \in \Sigma\) trivially, as \(b(E_x) = E_x\). Combining these, we conclude:
\[
\Sigma = \langle \tilde{a}_2, b \rangle.
\]

Finally, we describe the action of \(\Sigma\) on \(\mathbb{C} [E_x]\). The subgroup \(\Sigma\) acts via the homomorphism:
\[
\psi: \Sigma \to \mathbb{C}^\times,
\]
where \(\psi(\tilde{a}_2) = -1\) and \(\psi(b) = 1\). This gives the following isomorphism of \(\mathbb{C}\Gamma\)-modules:
\[
\mathcal{K}\kerlint(G) \cong \operatorname{Ind}_\Sigma^\Gamma \psi.
\]
\end{proof}

The final step is to determine the irreducible summands of this monomial module. We have the following theorem which is Part (d) of Theorem \ref{thm:main}.
\begin{theorem}
    The $\mathbb C[\Gamma]$-module $\mathcal{K}\kerlint(G)$ contains each of the irreducible $\mathbb C[\Gamma]$-modules of dimension 2 as a direct summand exactly once and has no other irreducible summands. 
\end{theorem}
\begin{proof}
By the previous lemma, we already know that $\mathcal{K}\kerlint(G)$ is induced from the 1-dimensional $\mathbb C[\Sigma]$-module $\psi$. To determine the irreducible summands, we first consider the inertia subgroup of $\psi$. Note that $\psi$ is not $\Gamma$-invariant. Indeed since $b^c = \tilde a_2b$, we have 
\[
({}^c\psi)(b) = \psi(b^c)\neq \psi(b).
\]
Its $\Gamma$-orbit consists of two elements, namely $\psi$ and ${}^c\psi$. In particular we also have
\[
\mathcal K\kerlint(G) \cong \ind_\Sigma^\Gamma {}^c\psi.
\]

Although it will not be used immediately, we determine ${}^c\psi$ more explicitly to refer to it later. It is straightforward to see that, following the above notation, the stabilizer of the line $\mathbb C[E_{xy}]$ is also equal to $\Sigma$. The action of generators of $\Gamma$ on $E_{xy}$ is given in the following table.
\begin{center}
    \begin{tabular}{c|cccc}
         & $a_1$ & $a_2$ & $b$ & $c$ \\ \hline
         $E_{xy}$ & $E_{x^{-1}y}$ & $E_{x^5y}$ & $E_{xzy}$ & $E_{x}$\\
           \end{tabular}
\end{center}
The character $\psi'$ of this action is given by 
\[
\psi'(\tilde a_2) = -1, \ \psi'(b)= -1
\]
It is clear that we have ${}^c\psi = \psi'$. 

Returning back to the inertia subgroup $I$ of $\psi$ in $\Gamma$, since $a_1, a_2$ are central and $b\in\Sigma=Z(\Sigma)$, the inertia subgroup $I$ of $\psi$ in $\Gamma$ contains $\langle a_1, a_2, b\rangle$.
But this is a maximal subgroup and $\psi$ is not $\Gamma$-invariant. Hence we must have equality.
\[
I = \langle a_1, a_2, b\rangle \cong C_2\times C_2\times C_{2^{m-2}}.
\]
 Since $I$ is abelian, the character $\psi$ extends to $I$ and by Clifford's Theory, $\ind_\Sigma^I\psi$ is the sum of distinct extensions of $\psi$ to $I$. More precisely, if $\tilde\psi$ is such an extension, then by the Gallagher Correspondence, the set of extensions, and  
hence the set of  irreducible characters of $I$ lying over $\psi$ is
\[
\text{Irr}(I\mid \psi) = \{ \lambda\cdot\tilde\psi \mid \lambda\in\widehat{I/\Sigma} \}.
\]
Here as usual, $(\lambda\cdot \tilde\psi)(g) = \lambda(g\Sigma)\tilde\psi(g)$.
Furthermore none of these extensions is $\Gamma$-invariant and hence by the Mackey Formula we get
\[
\res^\Gamma_I\ind_I^\Gamma \tilde\psi = \tilde\psi \oplus {}^c\tilde\psi
\]
Therefore by the Frobenius Reciprocity,
\begin{eqnarray*}
    \big\langle \ind_I^\Gamma\tilde\psi, \ind_I^\Gamma\tilde\psi\big\rangle_\Gamma &=&     \big\langle \tilde\psi, \res^\Gamma_I\ind_I^\Gamma\tilde\psi\big\rangle_I\\
    &=&     \big\langle \tilde\psi, \tilde\psi \oplus {}^c\tilde\psi\big\rangle_I = 1
\end{eqnarray*}
In particular, $\ind_I^\Gamma\tilde\psi$ is an irreducible $\mathbb C\Gamma$-module. Furthermore by the Clifford's Correspondence, the irreducible summands of $\ind_\Sigma^\Gamma\psi$ are in bijective correspondence with $\text{Irr}(I\mid \psi)$ and hence we get
\[
\mathcal K\kerlint(G) \cong \ind_\Sigma^\Gamma\psi \cong \bigoplus_{\chi\in\text{Irr}(I\mid \psi)}\ind_I^\Gamma \chi
\]
where the summands can be identified as pairwise non-isomorphic irreducible $\mathbb C\Gamma$-modules of dimension 2. Since the number of summands on the right hand side is equal to $|\Gamma|/8$. By Theorem \ref{thm:chars}, it is also the number of irreducible $\mathbb C\Gamma$-modules of dimension 2. Hence we conclude that irreducible constituents of $\ind_\Sigma^\Gamma\psi$ are exactly the ones of dimension 2 and each appears with multiplicity 1. 
\end{proof}

\begin{nothing}\label{sec:expbasis}
Finally we determine explicit bases for the irreducible summands of $\mathcal K\kerlint(G)$. For each $\chi\in\text{Irr}(I\mid \psi)$, let $e_\chi$ be the primitive idempotent of the group algebra $\mathbb CI$ corresponding to $\chi$. It is given by
\[
e_\chi = \frac{1}{|I|}\sum_{\iota\in I}\chi(\iota^{-1})\iota
\]
Recall that $\{ e_\chi\mid \chi\in \hat{I} \}$ is the set of pairwise orthonormal primitive idempotents summing up to 1. Hence the summand of $\ind_\Sigma^I\psi$ with character $\chi$ is equal to $e_\chi\cdot \ind_\Sigma^I\psi$. By Frobenius Reciprocity, the multiplicity of $\chi$ in $\ind_\Sigma^I\psi$ is 1, and hence it is generated by any nonzero vector it contains. To distinguish one, consider the product
\[
e_\chi\cdot E_x = \frac{1}{|I|}\sum_{\iota\in I}\chi(\iota^{-1})\iota(E_x)
\]
Here we regard $E_x$ as an element in $\ind_\Sigma^I\psi$ via the identification of $\mathbb C[E_x] = \psi$. Since $b$ fixes $E_x$ and $\chi(b) = 1$, the terms of the above sum corresponding to $b$ and $b\iota$ for any $\iota\in I$ are the same. Similarly, since $\tilde a_2\cdot E_x = -E_x$ and $\chi(\tilde a_2)= -1$, the terms corresponding to $\tilde a_2$ and $\tilde a_2\iota$, for any $\iota\in I$ are equal. Hence we may sum over the cosets of $\Sigma =\langle \tilde a_2, b\rangle$. After removing the terms with coefficient zero, we obtain the following nonzero vector
\[
F_{x, \chi} = \sum_{\iota\in [I/\Sigma]} \chi(\iota^{-1})E_{\iota(x)}.
\]
and hence get
\[
\ind_\Sigma^I \psi \cong \bigoplus_{\chi\in\text{Irr}(I\mid \psi)}\mathbb C[F_{x, \chi}].
\]
With this notation, we also have
\[
\res^\Gamma_I\ind_I^\Gamma \mathbb C[F_{x, \chi}] \cong \mathbb C[F_{x, \chi}] \oplus {}^c\big( \mathbb C[F_{x, \chi}]\big)
\]
where 
\[
{}^c\big( \mathbb C[F_{x, \chi}]\big) = \mathbb C[F_{xy, {}^c\chi}]
\]
One can easily check that both $F_{x, \chi} + F_{xy, {}^c\chi}$ and $F_{x, \chi} - F_{xy, {}^c\chi}$ are $\Gamma$-invariant and hence form a basis for $\ind_I^\Gamma \mathbb C[F_{x, \chi}]$. 
\end{nothing}
\begin{corollary}
    Assume the notation of Section \ref{sec:expbasis}. The simple summand of $\mathcal K\kerlint(G)$ corresponding to the irreducible character $\ind_I^\Gamma\chi$ of $\Gamma$ is given by 
    \[
\ind_I^\Gamma \mathbb C[F_{x, \chi}] = \mathbb C[F_{x, \chi} + F_{xy, {}^c\chi}] \oplus \mathbb C[F_{xy, \chi} - F_{xy, {}^c\chi}].
\]

\end{corollary}

\section{Calculation of Restriction Kernels V: $C_{p^m}\times C_p$-case}\label{sec:restker32}

For the last case, we consider the the groups $C_{p^m}\times C_p$ for $m\in\mathbb Z_{\ge 2}$. As in the previous case, we include the description of the automorphism group $\Gamma$ of $G$ in  \ref{appB}. 

\begin{nothing} By Section \ref{sec:restkerGen}, the restriction kernel $\mathcal{K}\kerlin(G)$ decomposed into a direct sum of vector spaces
\[
\mathcal K\kerlin(G) = \bigoplus_{gZ\in [\mathcal M/Z]} \mathcal K_{gZ}
\]
and the set 
\[
\mathcal B_p:=\big\{ E_{gz^i} \mid gZ\in [\mathcal M/Z], i=1, 2, \ldots, p-1\big\}
\]
is a $\mathbb C$-basis. Also consider the larger set
\[
\mathcal E_1:=\big\{ E_{g} \mid g\in\mathcal{M} \big\}.
\]
The group $\Gamma$ acts on the set $\mathcal E_1$ via automorphisms, that is, if $\gamma\in\Gamma$ and $g\in\mathcal{M}$, then $\gamma\cdot E_g = E_{\gamma(g)}$. We clearly have the following corollary.
\end{nothing}
\begin{corollary}
    The homomorphism of $\mathbb C[\Gamma]$-modules
    \[
    \mathbb C[\mathcal{E}_1] \to \mathcal{K}\kerlin(G), \, E_g\mapsto E_g
    \]
    is surjective. Its kernel is spanned by the orbit sums of the $Z$-action on $\mathcal M$, that is, by the sums
    \[
    E^+_{gZ} := \sum_{i=0}^{p-1}E_{gz^i}.
    \]
    as $g$ runs over a set $[\mathcal{M}/Z]$ of representatives of $Z$-orbits in $\mathcal M$.
\end{corollary}
\begin{nothing}
     The $Z$-orbits in $\mathcal{E}_1$ are permuted by the $\Gamma$-action since the subgroup $Z$ is a characteristic subgroup. In particular the $\Gamma$-action on $\mathcal{E}_1$ induces an action on the set $\mathcal{E}^+:=\{ E_{gZ}^+\mid g \in [\mathcal{M}/Z] \}$ and hence the above kernel is also a permutation $\mathbb C[\Gamma]$-module. Hence there is an isomorphism
    \[
    \mathcal{K}\kerlin(G) \cong \mathbb C[\mathcal{E}_1]/ \mathbb C[\mathcal{E}^+]
    \]
    of $\mathbb C[\Gamma]$-modules. We can make these permutation modules more precise.
\end{nothing}
\begin{lemma}
  \textrm{\textbf{(a)}}  The $\Gamma$-set $\mathcal{E}_1$ is transitive and $\Sigma:=\text{Stab}_\Gamma(E_x)$ is the subgroup of $\Gamma$ generated by the elements $b$ and $d$.

  \smallskip

  \textbf{(b)} The $\Gamma$-set $\mathcal{E}^+$ is transitive and $T:=\text{Stab}_\Gamma(E_x^+)$ is the subgroup $\Sigma\times \langle a^\omega\rangle$ of $\Gamma$.
\end{lemma}
\begin{proof}
    This lemma follows easily from the characterization of the generators of $\Gamma$ given in \ref{appB}. We only note that in Part (b), the element $a^\omega$ fixes $E_{xZ}^+$ because we have
    \[
    a^\omega(xz) = x^{s^\omega}z^{s^{\omega}} = x^{1+u}z^{1+u} = xz^{2}.
    \]
    \end{proof}
\begin{nothing}
By the lemma, we need to evaluate the quotient of two permutation $\mathbb C[\Gamma]$-modules. We write 
    \[
    \mathbb C[\mathcal{E}_1] \cong \ind_\Sigma^\Gamma \mathbb C = \ind_T^\Gamma \ind_\Sigma^T \mathbb C
    \quad
    \text{ and }
    \quad
    \mathbb C[\mathcal{E}^+] \cong \ind_T^\Gamma \mathbb C.
    \]
    Moreover since $\Sigma\unlhd T$ and $T/\Sigma\cong \langle a^\omega\rangle$, we have 
    \[
    \ind_\Sigma^T \mathbb C \cong \mathbb C[T/\Sigma] \cong \bigoplus_{V\in\text{Irr}(T/\Sigma)} \infl_{T/\Sigma}^T V.
    \]
    Substituting this in the above isomorphism,
    we obtain
    \[
    \mathbb C[\mathcal{E}_1] \cong \ind_T^\Gamma \mathbb C[T/\Sigma] \cong \bigoplus_{V\in\text{Irr}(T/\Sigma)} \ind_T^\Gamma\infl_{T/\Sigma}^T V.
    \]
    The term corresponding to the kernel $\mathbb C[\mathcal{E}^+]$ is the one parameterized by the trivial $\mathbb C[T/\Sigma]$-module $V=\mathbb C$.
    Hence we get
    \[
    \mathcal{K}\kerlin(G) \cong \bigoplus_{V\in\text{Irr}(T/\Sigma): V\not\cong \mathbb C} \ind_T^\Gamma\infl_{T/\Sigma}^T V
    \] 
    One still needs to decompose the induced modules 
    $\ind_T^\Gamma\infl_{T/\Sigma}^T V$. We denote the normalizer of $T$ in $\Gamma$ by $N$. It is clear that $a\in N$ and $c\not\in N$. Hence we must have 
    $N=\langle a, b, d\rangle$. We also have that $|\Gamma: N| = p$ and that $N$ is not normal in $\Gamma$. In particular, $N$ is self-normalizing. Moreover for any $V\in\text{Irr}(T/\Sigma)$, the irreducible $\mathbb C[T]$-module $\infl_{T/\Sigma}^T V$ is $N$-stable. In particular, $V$ extends to $N$ and $\ind_T^N V$ is the direct sum of distinct extensions of $V$. Writing $\text{Irr}(N|V)$ for the set of extensions of $V$, we obtain
    \[
    \mathcal{K}\kerlin(G) \cong \bigoplus_{V\in\text{Irr}(T/\Sigma): V\not\cong \mathbb C} \bigoplus_{W\in \text{Irr}(N|V) }\ind_N^\Gamma W
    \] 
    Finally we have the following result.
    \begin{lemma}\label{lem:irredps}
    With the above notation, the $\mathbb C[\Gamma]$-modules $\ind_N^\Gamma W$ are all irreducible of dimension $p$. 
    \end{lemma}
\begin{proof}
We determine the character of $\ind_N^\Gamma W$.
Suppose $V$ is the corresponding $\mathbb{C}[T/\Sigma]$-module and write $\chi_V$ for its 
character. Also write $\chi_W$ for the character of $W$. Since $W$ is an extension of 
$\infl_{T/\Sigma}^T V$ to $N$, in particular $W$ is a linear character of $N$.

We show $\langle \ind_N^\Gamma \chi_W, \ind_N^\Gamma \chi_W \rangle_\Gamma = 1$, which implies 
irreducibility. By Mackey's theorem, writing $N^{(s)} = N \cap s^{-1}Ns$, we have
\[
\langle \ind_N^\Gamma \chi_W, \ind_N^\Gamma \chi_W \rangle_\Gamma 
= \sum_{NsN \in N\backslash\Gamma/N} 
\langle \res^N_{N^{(s)}} \chi_W,\, \res^N_{N^{(s)}} {}^s\!\chi_W \rangle_{N^{(s)}}.
\]
The term $s = 1$ contributes $\langle \chi_W, \chi_W \rangle_N = 1$ since $W$ is irreducible.
For any $s \notin N$, write $s = c^k n'$ for some $k \in \{1, \ldots, p-1\}$ and $n' \in N$, 
and set $e = n'^{-1}bn' \in N$. From the relation $cb = a^{-\omega}bc$ in $\Gamma$ 
and the fact that $a^\omega \in Z(\Gamma)$, an easy induction gives $c^{-k}bc^k = a^{k\omega}b$,
and therefore
\[
s^{-1}es = n'^{-1}c^{-k}bc^k n' = n'^{-1}a^{k\omega}bn' = a^{k\omega}(n'^{-1}bn') = a^{k\omega}e \in N,
\]
so $e \in N^{(s)}$. Since $W$ is a linear character of $N$, conjugation does not change its 
value, so $\chi_W(e) = \chi_W(b) = 1$, where the last equality holds because 
$b \in \Sigma$ and $\chi_W|_T = \infl_{T/\Sigma}^T V$ is trivial on 
$\Sigma$. On the other hand,
\[
{}^s\!\chi_W(e) = \chi_W(s^{-1}es) = \chi_W(a^{k\omega}e) 
= \chi_W(a^\omega)^k \cdot \chi_W(e) = \zeta_p^k,
\]
where $\zeta_p = \chi_V(a^\omega\Sigma) \neq 1$ is a primitive $p$-th root of unity. Since 
$\zeta_p^k \neq 1$ for $k \in \{1,\ldots,p-1\}$, the linear characters $\chi_W$ and 
${}^s\!\chi_W$ disagree on $N^{(s)}$, and hence 
$\langle \res^N_{N^{(s)}} \chi_W, \res^N_{N^{(s)}} {}^s\!\chi_W \rangle_{N^{(s)}} = 0$.

Therefore $\langle \ind_N^\Gamma \chi_W, \ind_N^\Gamma \chi_W \rangle_\Gamma = 1$, 
and $\ind_N^\Gamma W$ is irreducible.
\end{proof}
    \end{nothing}

\begin{theorem}[Structure of \(\mathcal K\kerlin(G)\) for \(C_{p^m} \times C_p\)]
Let \(G \cong C_{p^m} \times C_p\) with \(m \geq 2\). Then
\[
\mathcal{K}\kerlin(G) \cong \bigoplus_{V\in\text{Irr}(T/\Sigma): V\not\cong \mathbb C} \bigoplus_{W\in \text{Irr}(N|V) }\ind_N^\Gamma W
\]
where each $\ind_N^\Gamma W$ is irreducible.  
\end{theorem}

\begin{nothing}
    Finally we note that not all irreducible $\mathbb C\Gamma$-modules of dimension $p$ appears as a direct summand of $\mathcal{K}\kerlin(G)$. As we will see in the example below, there is an automorphism of $\Gamma$ acting non-trivially on the center and hence producing more irreducibles of degree $p$. 

    With this remark, we have completed the proof of Theorem \ref{thm:main}.
    
\end{nothing}

\appendix

\section{The group $\Aut(C_{2^m}\times C_2)$ and its Irreducibles}\label{app:evenp}

This appendix describes the automorphism group \(\Gamma = \operatorname{Aut}(C_{2^m} \times C_2)\) for $m\ge 3$ and its irreducible representations. These results are used in Section~8 to analyze the \(\mathbb C[\Gamma]\)-module structure of \(\mathcal{K}\kerlint(G)\) for \(G = C_{2^m} \times C_2\).

\begin{nothing} {\textbf{The group.}} 
Fix $m\in\mathbb Z_{\ge 3}$ and write $G = C_{2^m}\times C_2$. 
We give a presentation of the group $\Gamma:= \Aut(G)$ following \cite{BC}.
Write $G = \langle x, y\mid x^{2^m} = y^2 = 1, xy=yx\rangle$ and let $z = x^{2^{m-1}}$ be the unique element of order 2 in $\langle x\rangle$. We may regard the elements of $G$ as column vectors via identifying $x^i\cdot y^j$ by $\begin{pmatrix}
    x^i\\ y^j
\end{pmatrix}$ With this notation, any automorphism of $G$ may be described using a matrix notation as follows: 

Let  $a:= \begin{pmatrix}
    \delta_1& \delta_2\\ \delta_3&\delta_4
\end{pmatrix}$ be a $2\times 2$-matrix where: 
\[\delta_1\in\Aut(\langle x\rangle), \delta_2\in\Hom(\langle y\rangle, \langle x\rangle), \delta_3\in \Hom(\langle x\rangle, \langle y\rangle) \text{ and } \delta_4\in\Aut(\langle y\rangle).\]
Then, for any $(h, k)\in G$, we have
\[
a(h, k) = \begin{pmatrix}
    \delta_1& \delta_2\\ \delta_3&\delta_4
\end{pmatrix}\begin{pmatrix}
    h\\ k
\end{pmatrix} = (\delta_1(h)\delta_2(k), \delta_3(h)\delta_4(k))
\]
which is clearly an endomorphism of $G$. Now define the homomorphisms
\begin{eqnarray*}
    \alpha_1 : \langle x\rangle \to \langle x\rangle, && x\mapsto x^{-1},\\
        \alpha_2 : \langle x\rangle \to \langle x\rangle, && x\mapsto x^{5},\\
    \beta : \langle y\rangle \to \langle x\rangle, && y\mapsto z,\\
    \gamma : \langle x\rangle \to \langle y\rangle, && x\mapsto y.
\end{eqnarray*}
Corresponding to each of these homomorphisms, we construct the following matrices
\[
a_1:= \begin{pmatrix}
    \alpha_1 & 0\\ 0& 1
\end{pmatrix},\ 
a_2:= \begin{pmatrix}
    \alpha_2& 0\\ 0&1
\end{pmatrix},\ 
b:= \begin{pmatrix}
    1& \beta\\ 0&1
\end{pmatrix},\ 
c:= \begin{pmatrix}
    1& 0\\ \gamma &1
\end{pmatrix}
\]
which clearly induce automorphisms of $G$. 
For future use, we also list the evaluation of each of the above automorphisms on the generators $x$ and $y$ of $G$.
\begin{center}
    \begin{tabular}{c|cccc}
         & $a_1$ & $a_2$ & $b$ & $c$ \\ \hline
         $x$ & $x^{-1}$ & $x^5$ & $x$ & $xy$\\
         $y$ & $y$ & $y$ & $zy$ & $y$\\
    \end{tabular}
\end{center}
By \cite[Section 3]{BC}, the automorphism group \(\Gamma\) is generated by the elements \(a_1, a_2, b, c\), with the following relations:
\begin{align*}
\Gamma = \langle a_1, a_2, b, c \mid a_1^2 = a_2^{2^{m-2}} = b^2 = c^2 = 1, \, b^c = a_2^{2^{m-3}} b,\\ \text{others commute} \rangle.
\end{align*}
Some properties of \(\Gamma\) include:
\begin{itemize}
    \item \(\Gamma\) has order \(2^{m+1} = |G|\).
    \item The center \(Z(\Gamma)\) is generated by \(a_1, a_2\).
    \item The commutator subgroup \(\Gamma'\) is the cyclic subgroup generated by \(a_2^{2^{m-3}}\).
    \item There are isomorphisms \begin{eqnarray*}
\Gamma 
&\cong& \big( (\langle a_2\rangle \times \langle b\rangle)\rtimes \langle c\rangle \big) \times \langle a_1\rangle \\
&\cong & \big( (C_{2^{m-2}}\times C_2)\rtimes C_2 \big) \times C_2 
\end{eqnarray*}
\end{itemize}

\end{nothing}
\begin{nothing} {\textbf{The characters.}}
Next we determine the degrees of irreducible characters of $\Gamma$. By the above given relations and properties, it is clear that 
\[
I:= \langle a_1, a_2, b\rangle \cong C_2\times C_2\times C_{2^{m-2}}.
\]
 is a normal abelian subgroup of index 2 in $\Gamma$. Hence by a well-known result from character theory, every irreducible character of $\Gamma$ has degree 1 or 2. On the other hand, the number of linear characters of $\Gamma$ is equal to $|\Gamma: \Gamma'| = |\Gamma|/2$. Thus
\[
|\Gamma | = \sum_{\chi\in\text{Irr}(\Gamma)}\chi(1)^2 = |\Gamma|/2 + 
\sum_{\chi\in\text{Irr}(\Gamma): \chi(1) = 2}4
\]
Hence the group $\Gamma$ has $|\Gamma|/8$ irreducible characters of degree 2. 

One can say more about the degree 2 irreducible characters of $\Gamma$. Let $\chi\in\text{Irr}(G)$ with $\chi(1)=2$. Also let $\theta\in\text{Irr}(I)$ be an irreducible constituent of the restriction of $\chi$ to $I$. Then by Frobenius reciprocity, we must have
\[
\ind_I^\Gamma\theta = \chi
\]
since $|\Gamma: I| = 2 = \chi(1)$. It is straightforward to show that the character $\theta$ above cannot be $\Gamma$-stable and moreover any irreducible character $\chi\in\text{Irr}(\Gamma)$ of degree 2 is obtained in this way. 
As a result we obtain the following result.
\begin{theorem}\label{thm:chars}
    The group $\Gamma = \Aut(G)$ has 
    \begin{enumerate}
        \item $|\Gamma|/2$ irreducible characters of degree 1 and
        \item $|\Gamma|/8$ irreducible characters of degree 2.
    \end{enumerate}
    Moreover any irreducible character of $\Gamma$ of degree 2 is the induction of some $\Gamma$-non-stable irreducible character of the subgroup $I$ of $\Gamma$.
\end{theorem}
\end{nothing}

\section{The group $\Aut(C_{p^m}\times C_p)$}\label{appB}
Fix an odd prime $p$, a positive integer $m\ge 2$ and write $G =\langle x\rangle \times \langle y\rangle$ where $x$ is of order $p^m$ and $y$ is of order $p$. 
As in the above case, the group $\Gamma= \Aut(G)$ is generated by four elements, written $a, b, c$ and $d$. We list the evaluation of each of these automorphisms on the generators $x$ and $y$ of $G$. Let $s$ be a primitive $p^m$-th root of unity and put $u = p^{m-1}$. Then
\begin{center}
    \begin{tabular}{c|cccc}
         & $a$ & $b$ & $c$ & $d$ \\ \hline
         $x$ & $x^{s}$ & $x$ & $xy$ & $x$\\
         $y$ & $y$ & $x^uy$ & $y$ & $y^s$\\
    \end{tabular}
\end{center}
To describe the relations between the generators, we introduce two more numbers. Let $t$ be the multiplicative inverse of $s$ in the group $\mathbb Z/p^m\mathbb Z$. Also we choose $\omega$ such that $s^\omega \equiv 1 + u$ (mod $p^m$). With this notation, (also choosing $n=1$), \cite[Theorem 2.5]{BC} gives the following presentation of $\Gamma$.
\begin{eqnarray*}
    \Gamma = &\langle a, b, c, d\mid a^{\Phi(p^m)} = b^p = c^p = d^{p-1} =1, b^a = b^t, b^d =b^s, \\ & c^a=c^s, c^d = c^t, a^d = a, cb = a^{-\omega}bc \rangle. 
    \end{eqnarray*}
    Some properties of \(\Gamma\) include:
\begin{itemize}
    \item \(|\Gamma| = (p-1)^2 p^{m+1}\).
    \item \(\Gamma\) has a cyclic center generated by \(ad\) and $a^\omega\in Z(\Gamma)$.
    \item The commutator subgroup is generated by \(b\) and \(c\) and it intersects $\langle a\rangle\langle d\rangle$ at the cyclic subgroup generated by $a^\omega$.
\end{itemize}

\section{Examples}

In this section we evaluate composition factors for small primes. These examples illustrate some theoretical results in the paper and give an insight into the structure of the restriction kernel.

\begin{nothing}{\bf Examples for \(p = 2\).} 
The table below summarizes the composition factors for some small minimal groups \(G\).
\end{nothing}

\begin{table}[h!]
\centering
\begin{tabular}{|c|c|}
\hline
\textbf{Minimal Group} & \textbf{Composition Factors} \\
\hline
\(C_2\) & \(S_{C_2, 1}\) \\
\(C_2 \times C_2\) & \(S_{C_2 \times C_2, 1}\) \\
\(C_4\) & \(S_{C_4, \phi_1}, S_{C_4, \phi_2}\) \\
$C_8$ & $S_{C_8, \phi_0}, S_{C_8, \phi_1}, S_{C_8, \phi_2}, S_{C_8, \phi_3}$ \\
$C_4\times C_2$ & $S_{C_4\times C_2, V}$\\
$C_8\times C_2$ & $S_{C_8\times C_2, \chi_1}, S_{C_8\times C_2, \chi_2}$\\
\hline
\end{tabular}
\end{table}

    Details are as follows. We use both Theorem \ref{thm:main} and the proof of the relevant part when necessary.
    \begin{itemize}
        \item If $G=C_2$, then $\kerlin(G) \cong \mathbb C$, so the only composition factor with minimal group $C_2$ is $S_{C_2, 1}$.
        \item If $G=C_2\times C_2$, then $S_{C_2\times C_2, 1}$ is the only composition factor with this minimal group. Note that this is the one isomorphic to the functor $\mathbb CD$ of the Dade group.
        \item If $G=C_4$, then $\Aut(C_4)\cong C_2$. Write $\widehat{\Aut(G)} = \{\phi_1, \phi_2\}$. Then both $S_{C_4, \phi_1}$ and $S_{C_4, \phi_2}$ are composition factors with multiplicity 1.
        \item If $G=C_8=\langle g\rangle$, then $\Aut(G)\cong C_2\times C_2$. Write $\Aut(G) = \{ 1, \alpha, \beta, \alpha\beta \}$ where $\alpha(g) = g^3$ and $\beta(g) = g^5$. 
        Also construct the character table of $\widehat{\Aut(G)}$:
        $$
\begin{array}{c|rrrr}
  C_2\times C_2&\rm1&\alpha&\beta&\alpha\beta\cr
\hline
  \phi_{0}&1&1&1&1\cr
  \phi_{1}&1&1&-1&-1\cr
  \phi_{2}&1&-1&1&-1\cr
  \phi_{3}&1&-1&-1&1\cr
\end{array}
$$
Then by Theorem \ref{thm:main}, for each $\phi_i, i=0, 1, 2, 3$, the simple functor $S_{C_8, \phi_i}$ appear as a composition factor with multiplicity 1. Note that the character corresponding to the pair $(2, \phi)$ where $\phi$ is the non-trivial irreducible character of $\mathbb C[\Aut(C_4)]$ is $\phi_2$.

\item The case $G = C_4\times C_2$ is already covered in Theorem \ref{thm:main}(d). 
\item If $G=C_8\times C_2$, then $\Aut(G) = C_2\times D_8$. The subgroup $I =\langle a_1, a_2, b\rangle$ has order 8. The characters lying over $\psi$ are given as follows.
$$
\begin{array}{c|rrrrrrrr}
  (C_2)^3&\rm1&a_1&a_2&b&a_1a_2&a_1b&a_2b&a_1a_2b\cr
\hline
  \phi_{1}&1&1&-1&1&-1&1&-1&-1\cr
  \phi_{2}&1&-1&-1&1&1&-1&-1&1\cr
\end{array}
$$
Then both $\chi_1=\ind_I^{\Aut(G)}\phi_1$ and $\chi_2=\ind_I^{\Aut(G)}\phi_2$ are irreducible, so we obtain the last two composition factors in the above table. Without further details, we include the characters $\chi_1$ and $\chi_2$.
$$
\begin{array}{c|rrrrrrrrrr}
  \rm C_2\times D_8&\rm1&a_1&a_2&a_1a_2&b&c&a_1b&a_1c&bc&a_1bc\cr
\hline
  \chi_{1}&2&2&-2&-2&0&0&0&0&0&0\cr
  \chi_{2}  &2&-2&-2&2&0&0&0&0&0&0\cr
\end{array}
$$
    \end{itemize}

\begin{nothing}{\bf Examples for $p=3$.}
    For the second example assume $p=3$. This case has some simplifications but it is still instructive. We have the following list.
\begin{table}[h!]
\centering
\begin{tabular}{|c|c|}
\hline
    Minimal group & Composition factor \\ \hline
    $C_3$ & $S_{C_3, 1}$ \\
    $C_3 \times C_3$ & $S_{C_3\times C_3, 1}, S_{C_3\times C_3, \chi}$ \\
    $C_9$ & $S_{C_9, \phi_i}, \, i=0\ldots 5$ \\
    $C_9 \times C_3$ & $S_{C_9\times C_3, \chi_{16}}, S_{C_9\times C_3, \chi_{17}},$ \\
    & $S_{C_9\times C_3, \chi_{18}}, S_{C_9\times C_3, \chi_{19}}$ \\
    \hline
\end{tabular}
\end{table}

\begin{itemize}
    \item We leave the details of the cases $G=C_3$ and $G=C_9$ to the reader.
    \item If $G=C_3\times C_3=\langle x\rangle\times \langle y\rangle$, then $\Aut(G) = {\rm GL}(2, 3):= \Gamma$. Note that $\Gamma$ is of order $48$ and it has a unique irreducible character $\chi$ of degree 4. We show how it is constructed using the above proof. By Theorem \ref{thm:cpcpfinal2}, the simple functor $S_{C_3\times C_3, 1}$ is a composition factor and the number of non-trivial direct summands of $\mathcal{K}\kerlin(G)$ is $|\Aut(C_3)|-1 = 1$. This summand is induced from a linear character of the Borel subgroup.
    Now, the Borel subgroup $B$ of $\Gamma$ is of order 12 and it is generated by the following 3 matrices.
    \[
    \alpha = \begin{pmatrix}
        1&1\\0&1
    \end{pmatrix},\, \beta=\begin{pmatrix}
        2&0\\0&1
    \end{pmatrix},\, \gamma=\begin{pmatrix}
        1&0\\0&2
    \end{pmatrix}
    \]
    This subgroup is isomorphic to the Dihedral group of order 12. 
    We let $\phi$ be the unique linear character of $B$ such that $\phi(\beta) = -1$ and $\phi(\gamma) = 1$. Then 
    \[
    \chi = \ind_B^\Gamma\phi
    \]
    is irreducible of degree 4 and hence the simple functor $S_{C_3\times C_3, \chi}$ is a composition factor of $\kerlin$. 

    Note that when $p>3$, the group $\hat\Gamma_x$ has $p-2>1$ non-trivial irreducible characters and for each of them, we will have an irreducible character of $\Gamma$ induced from $B$. 
    \item Finally, let $G= C_9\times C_3$. We choose $s=2$ and hence $t=5$ and $\omega = 2$. Thus the automorphism group $\Gamma$ can be presented as follows.
    \begin{eqnarray*}
    \Gamma = &\langle a, b, c, d\mid a^6 = b^3 = c^3 = d^2 =1, b^a = b^5, b^d =b^2, \\ & c^a=c^2, c^d = c^5, a^d = a, cb = a^4bc \rangle. 
    \end{eqnarray*}
    We include the part of the character table of $\Gamma$ containing degree $3$ characters. See \cite{Dokchitser} for the complete character table.
\end{itemize}
   
$$
\begin{array}{c|crrrcccccccccc}
  \rm class&\rm1& a^3d&a^3& d&a^2&a^4&\rm3*&ad&a^5d&\rm6*&\rm6G&\rm6H&\rm6I&\rm6J\cr
  \rm size&1&1&9&9&1&1&6&1&1&6&9&9&9&9\cr
\hline
\rho_{13}&3&3&-1&-1&3\zeta_3&3\zeta_3^2&0&3\zeta_3^2&3\zeta_3&0&\zeta_6^5&\zeta_6&\zeta_6&\zeta_6^5\cr
\rho_{14}&3&3&-1&-1&3\zeta_3^2&3\zeta_3&0&3\zeta_3&3\zeta_3^2&0&\zeta_6&\zeta_6^5&\zeta_6^5&\zeta_6\cr
\rho_{15}&3&-3&1&-1&3\zeta_3&3\zeta_3^2&0&-3\zeta_3^2&-3\zeta_3&0&\zeta_6^5&\zeta_6&\zeta_3^2&\zeta_3\cr
\rho_{16}&3&-3&-1&1&3\zeta_3^2&3\zeta_3&0&-3\zeta_3&-3\zeta_3^2&0&\zeta_3^2&\zeta_3&\zeta_6^5&\zeta_6\cr
\rho_{17}&3&-3&-1&1&3\zeta_3&3\zeta_3^2&0&-3\zeta_3^2&-3\zeta_3&0&\zeta_3&\zeta_3^2&\zeta_6&\zeta_6^5\cr
\rho_{18}&3&3&1&1&3\zeta_3^2&3\zeta_3&0&3\zeta_3&3\zeta_3^2&0&\zeta_3^2&\zeta_3&\zeta_3&\zeta_3^2\cr
\rho_{19}&3&3&1&1&3\zeta_3&3\zeta_3^2&0&3\zeta_3^2&3\zeta_3&0&\zeta_3&\zeta_3^2&\zeta_3^2&\zeta_3\cr
\rho_{20}&3&-3&1&-1&3\zeta_3^2&3\zeta_3&0&-3\zeta_3&-3\zeta_3^2&0&\zeta_6&\zeta_6^5&\zeta_3&\zeta_3^2\cr
\end{array}
$$

Here we note that the classes $3*$ and $6*$ are representing the four conjugacy classes of elements of order $3$ and $6$, respectively. We exclude them to save some space. We also make changes to the names of some classes according to the above presentation. As usual, $\zeta_3$ (resp. $\zeta_6$) denotes a primitive third (resp. sixth) root of unity.

Arguing as in the proof of Lemma \ref{lem:irredps}, one can prove that if $V$ is a summand of $\mathcal{K}\kerlin(G)$, then its character $\chi_V$ satisfies $\chi_V(d) = 1.$ 
Hence we must have that
\[
\chi_{\mathcal{K}\kerlin(G)} = \chi_{16} + \chi_{17} + \chi_{18} + \chi_{19}.
\]
\end{nothing}

\section{A Result on Composition Factors}

We append this section to describe the theorem from \cite{BCK} which is used to conclude completeness of composition factors. Although the original statement is for modules over arbitrary Green biset functors, we state its special case for biset functors. It is included for self-completeness. 

\begin{theorem}
Let \( F \) be a biset functor over a field \( k \) and let \( G \) be a finite group. Suppose that \( \dim_k F(G) < \infty \), and that for every \( \mathcal E_G \)-submodule \( M \subseteq F(G) \), one has
\[
M = (M \cap \mathcal K_F(G)) + I_G M.
\]
Then, for every simple \( k[\out(G)]\)-module \( V \), 
the following numbers are equal:
\begin{enumerate}
    \item The multiplicity $[F : S_{G,V}]$ of $S_{G, V}$ as a composition factor of $F$.
    \item The multiplicity of $V$ as a composition factor of the $\mathcal E_G$-module $F(G)$.
    \item The multiplicity of $V$ as a composition factor of the $k[\out(G)]$-module $\mathcal KF(G)$.
\end{enumerate}
\end{theorem}

\begin{proof} The equality of the first two numbers is well-known. We prove the equality of the last two numbers.
Since \( \dim_k F(G) < \infty \), there exists an \( \mathcal E_G \)-composition series
\[
0 = M_0 \subset M_1 \subset \cdots \subset M_{n-1} \subset M_n = F(G)
\]
of \( F(G) \). Set \( K := \mathcal KF(G) \), and consider the induced series
\[
0 = (M_0 \cap K) \subseteq (M_1 \cap K) \subseteq \cdots \subseteq (M_{n-1} \cap K) \subseteq (M_n \cap K) = K
\]
of \( \mathcal E_G \)-submodules of \( K \). Let \( V \) be a simple \( k[\out(G)]\)-module, and let \( i \in \{1, \ldots, n\} \). We claim that if \( M_i / M_{i-1} \cong V \), then \( (M_i \cap K)/(M_{i-1} \cap K) \cong V \). This implies that \( [F(G) : V] \leq [K : V] \). But clearly, \( [K : V] \leq [F(G) : V] \), and we obtain equality. 

To prove the claim, observe that
\[
\frac{M_i \cap K}{M_{i-1} \cap K} = \frac{M_i \cap K}{(M_i \cap K) \cap M_{i-1}} \cong \frac{(M_i \cap K) + M_{i-1}}{M_{i-1}} \subseteq \frac{M_i}{M_{i-1}} \cong V.
\]
It therefore suffices to show that the left-hand side is nonzero. Assume by contradiction that \( M_i \cap K = M_{i-1} \cap K \). Then,
\[
M_i = (M_i \cap K) + I_G M_i \subseteq (M_i \cap K) + M_{i-1} = M_{i-1},
\]
contradicting \( M_i / M_{i-1} \cong V \neq 0 \). Here, we used the assumption that \( V \) is annihilated by \( I_G \), and the hypothesis of the theorem.
\end{proof}

\end{document}